\documentclass[a4paper,english]{amsart}
\usepackage[T1]{fontenc}
\usepackage[latin9]{inputenc}
\usepackage{color}
\usepackage{amsthm}
\usepackage{amstext}
\usepackage{graphicx}
\usepackage{amssymb}
\usepackage{esint}

\makeatletter

\pdfpageheight\paperheight
\pdfpagewidth\paperwidth

\numberwithin{equation}{section}
\numberwithin{figure}{section}
\theoremstyle{plain}
\newtheorem{thm}{Theorem}[section]
  \theoremstyle{plain}
  \newtheorem{prop}[thm]{Proposition}
 \theoremstyle{definition}
  \newtheorem{example}[thm]{Example}
  \theoremstyle{definition}
  \newtheorem{defn}[thm]{Definition}
  \theoremstyle{remark}
  \newtheorem{rem}[thm]{Remark}
  \theoremstyle{plain}
  \newtheorem{cor}[thm]{Corollary}
  \theoremstyle{plain}
  \newtheorem{lem}[thm]{Lemma}
  \theoremstyle{definition}
  \newtheorem*{example*}{Example}

\usepackage{amssymb,amsmath}
 
\usepackage{textcomp}
\usepackage{amsthm}
\usepackage{latexsym}

\usepackage{float}
\usepackage[latin9]{inputenc}

 \usepackage{color}

\usepackage{fancybox}
\usepackage{array}
\usepackage{fancyhdr}
\usepackage{pgf}
\usepackage{bbm}

\DeclareMathOperator{\supp}{supp}

\DeclareMathOperator{\card}{card}

\DeclareMathOperator{\spn}{span}
\DeclareMathOperator{\cl}{cl}

\DeclareMathOperator{\id}{id}

\newcommand{\Z}{\mathbb{Z}}

\newcommand{\U}{\mathcal{U}}
\newcommand{\T}{\mathcal{T}}
\newcommand{\R}{\mathbb {R}}

\newcommand{\C}{\mathbb {C}}

\newcommand{\N}{\mathbb {N}}

\newcommand{\1}{\mathbbm{1}}

\renewcommand{\epsilon}{\varepsilon}
\renewcommand{\emptyset}{\varnothing}

\@ifundefined{showcaptionsetup}{}{%
 \PassOptionsToPackage{caption=false}{subfig}}
\usepackage{subfig}
\makeatother

\usepackage{babel}

\begin{document}
\selectlanguage{english}

\title{Multiresolution analysis for Markov Interval Maps}

\author{Jana Bohnstengel and Marc Kesseböhmer }
\begin{abstract}
We set up a multiresolution analysis on fractal sets derived from
limit sets of Markov Interval Maps. For this we consider the $\Z$-convolution
of a non-atomic measure supported on the limit set of such systems
and give a thorough investigation of the space of square integrable
functions with respect to this measure. We define an abstract multiresolution
analysis, prove the existence of mother wavelets, and then apply these
abstract results to Markov Interval Maps. Even though, in our setting
the corresponding scaling operators are in general not unitary we
are able to give a complete description of the multiresolution analysis
in terms of multiwavelets. 
\end{abstract}

\date{July 1, 2011}

\address{Fachbereich 3 - Mathematik und Informatik, Universität Bremen, Bibliothekstrasse
1, 28359 Bremen, Germany}

\email{bohni@math.uni-bremen.de, mhk@math.uni-bremen.de}

\keywords{multiwavelets, multiresolution analysis, Markov interval maps}

\maketitle
\tableofcontents{}

\section{Introduction and main results}

The main aim of this paper is to construct a wavelet basis on limit
sets of Markov Interval Maps (MIM) in the unit interval translated
by $\Z$. In this way we extend the results in \cite{DJ03,BoKe10},
where wavelet bases with respect to singular measures were provided.
For this we first prove that a MIM gives rise to a multiresolution
analysis (MRA) and study the particular case where the underlying
measure is Markovian. This MRA is then reformulated in an abstract
way allowing us to prove the existence of a wavelet basis in this
abstract setting. 

In the case of a fractal given by an iterated function system (IFS)
on $[0,1]$ there are several approaches to construct a wavelet basis
on the $L^{2}$-space with respect to a suitable singular measure
which is supported on a so-called enlarged fractal. The enlarged fractal
is derived from the original fractal by first mapping scaled copies
of it to each gap interval and then by taking the union of translats
by $\Z$ defining a dense set in $\R$. In \cite{DJ03} the authors
construct a wavelets basis for fractals on self-similar Cantor sets,
i.e. sets that are given by affine IFS with the same scaling factor
$1/N$, $N\geq2$, for all $p\leq N$ branches. They consider the
$L^{2}$-space with respect to $\mu$, the $\delta$-dimensional Hausdorff
measure restricted to the enlarged fractal, where $\delta$ denotes
the dimension of the Cantor set. In this situation the analysis depends
on the two unitary operators $U$ and $T$, where $U$ denotes the
\emph{scaling operator} given by $Uf:=\sqrt{p}f\left(N\cdot\right)$
and $T$ denotes the \emph{translation operator} given by $Tf:=f(\cdot-1)$
for $f\in L^{2}(\mu)$. Furthermore, a natural choice for a father
wavelet $\varphi$ is the characteristic function on the original
fractal. The authors show that for a family of closed subspaces $\left(V_{j}\right)_{j\in\Z}$
of $L^{2}(\mu)$ the following six conditions are satisfied. 
\begin{itemize}
\item $\dots\subset V_{-2}\subset V_{-1}\subset V_{0}\subset V_{1}\subset V_{2}\subset\cdots$, 
\item $\text{cl}\bigcup_{j\in\mathbb{Z}}V_{j}=L^{2}(\mu)$, 
\item $\bigcap_{j\in\mathbb{Z}}V_{j}=\{0\}$, 
\item $V_{j+1}=UV_{j}$, $j\in\mathbb{Z}$,
\item $\left\{ T^{n}\varphi:\; n\in\mathbb{Z}\right\} $ is an orthonormal
basis in $V_{0}$,
\item $U^{-1}TU=T^{N}$.
\end{itemize}
These observations allow the authors to construct a wavelets basis
for $L^{2}(\mu)$ explicitly in terms of certain filter functions. 

In \cite{BoKe10} we generalize this approach by allowing conformal
IFS satisfying the open set condition on $[0,1]$. We choose the measure
of maximal entropy supported on the fractal and this measure is extended
to a measure $\mu$ supported on the enlarged fractal in $\R$. Then
similarly as in \cite{DJ03} we construct the wavelet basis via MRA
in terms of the unitary scaling operator $U$ and the unitary translation
operator $T$. Again via filter functions the mother wavelets $\psi_{i}$,
$i\in\left\{ 1,\dots,N-1\right\} $ are defined such that $\left\{ U^{n}T^{k}\psi_{i}:n,k\in\Z,i\in\{1,\dots,N-1\}\right\} $
provides an orthonormal basis of $L^{2}(\mu)$. 

Here, our aim is to extend the construction of wavelet bases with
respect to fractal measures to the construction of wavelet bases on
the by $\Z$ translated limit set of a Markov Interval Map (MIM).
A Markov Interval Map consists of a family $\left(B_{i}\right)_{i=0}^{N-1}$
of closed subintervals in $[0,1]$ with disjoint interior, and a function
$F:\bigcup_{i\in\underline{N}}B_{i}\rightarrow[0,1]$, such that $F|{}_{B_{i}}$
is expanding and $C^{1}$, $i\in\underline{N}:=\left\{ 0,\ldots,N-1\right\} $
and such that $F\left(B_{i}\right)\cap B_{j}\not=\emptyset$ implies
$B_{j}\subset F\left(B_{i}\right)$. Its (fractal) \emph{limit set}
is given by $X:=\bigcap_{n=0}^{\infty}F^{-n}I$, where $I:=\bigcup_{i\in\underline{N}}B_{i}$.
By considering its inverse branches $\tau_{i}:=\left(F|_{B_{i}}\right)^{-1}$,
$i\in\underline{N}$, we obtain a Graph Directed Markov System (see
\cite{MU03}) with incidence matrix $A=\left(A_{ij}\right)_{i,j\in\underline{N}}$,
where $A_{ij}=1$ if $F\left(B_{i}\right)\supset B_{j}$ and $0$
otherwise. For the precise definition see Definition \ref{def:MFS}
and for an explicit example of an MIM see Example \ref{exa:-Transformation}
where we consider the $\beta$-transformation. The limit set $X$
is --~up to a countable set where it is finite-to-one~-- homeomorphic
to the topological Markov chain $\Sigma_{A}:=\{\omega=(\omega_{0},\omega_{1},\dots)\in\underline{N}^{\N}:\, A_{\omega_{i}\omega_{i+1}}=1\,\forall i\geq0\}$.
For the definition of the canonical coding map $\pi$ from $\Sigma_{A}$
to $X$ see (\ref{eq:codingmap}).

Given a Markov measure $\widetilde{\nu}$ on the shift space $\Sigma_{A}$
with a probability vector $\left(p_{i}\right)_{i\in\underline{N}}$
and stochastic matrix $\left(\pi_{ij}\right)_{i,j\in\underline{N}}$,
we consider the probability measure $\nu:=\widetilde{\nu}\circ\pi^{-1}$,
to which we also refer as $\nu$ a Markov measure. The $\mathbb{Z}$-convolution
(by translations) of $\nu$ is given by \[
\nu_{\Z}:=\sum_{k\in\Z}\nu(\cdot-k).\]
Similar to the construction in \cite{BoKe10} we introduce the scaling
operator \begin{equation}
Uf(x):=\sum_{k\in\Z}\sum_{j\in\underline{N}}\sum_{i\in\underline{N}}\sqrt{\frac{p_{i}}{p_{j}\pi_{ji}}}\cdot\mathbbm{1}_{[ji]}(x-k)\cdot f(\tau_{j}^{-1}(x-k)+j+Nk)\label{eq:def U}\end{equation}
 and the translation operator \begin{equation}
Tf(x):=f(x-1)\label{eq:Def T}\end{equation}
for $f\in L^{2}(\nu_{\Z})$ and $x\in\R$, where $[ji]\subset\R$,
$i,j\in\underline{N}$, denotes a cylinder set (see Section \ref{sec:Setting-(IFS-with}).
It is important to note that in contrast to the construction of the
scaling operator for IFS the operator $U$ is in general \emph{not
unitary.} Nevertheless, we have the following properties.
\begin{prop}
\label{pro:eigenschaften UT-1}Let $\left(\varphi_{i}\right)_{i\in\underline{N}}$
denote a family of father wavelets given by $\varphi_{i}:=\sqrt{\nu([i])}^{-1}\mathbbm{1}_{[i]}$,
$i\in\underline{N}$. The translation operator $T$ and the scaling
operator $U$ satisfy the following properties.
\begin{enumerate}
\item \label{enu:Prop U (1)}$TU=UT^{N}$,
\item \label{enu:,Prop U (2)}$\varphi_{i}=U\sum_{j\in\underline{N}}\sqrt{\pi_{ij}}T^{i}\varphi_{j}$,
$i\in\underline{N}$,
\item \label{enu:,Prop U (3)}$\langle T^{k}\varphi_{i}|T^{l}\varphi_{j}\rangle=\delta_{(k,i),(l,j)}$,
$k,l\in\Z$, $i,j\in\underline{N}$, 
\item \label{enu:,Prop U (4)}$UU^{*}=\id$,
\item \label{enu:-Prop U (5)}$U^{*}U=\id$ if and only if $A_{ij}=1$ for
all $i,j\in\underline{N}.$ 
\end{enumerate}
\end{prop}
For an explicit formula of $U^{*}$ see (\ref{eq:U*}). As an example
for this setting we consider the $\beta$-transformation. 
\begin{example}
[$\beta$-Transformation]\label{exa:-Transformation} Let $\beta:=\frac{1+\sqrt{5}}{2}$
denote the golden mean. Then the $\beta$-transformation is given
by $F:\left[0,1\right]\to\left[0,1\right]$, $x\mapsto\beta x\mod1$
(see Figure \ref{fig:-betatransform} for the graph of $F$). This
map can be considered as a MIM as follows. In this case we have $X:=[0,1]$
and the inverse branches are $\tau_{0}(x):=\frac{x}{\beta}$, $x\in[0,1]$,
and $\tau_{1}(x):=\frac{x+1}{\beta}$, $x\in[0,\beta-1]$. We may
choose the two intervals $B_{0}:=[0,\beta-1]$ and $B_{1}:=[\beta-1,1]$
and the corresponding transition matrix is then given by $A:=\left(\begin{array}{cc}
1 & 1\\
1 & 0\end{array}\right)$. 

\begin{figure}
\includegraphics[scale=0.3]{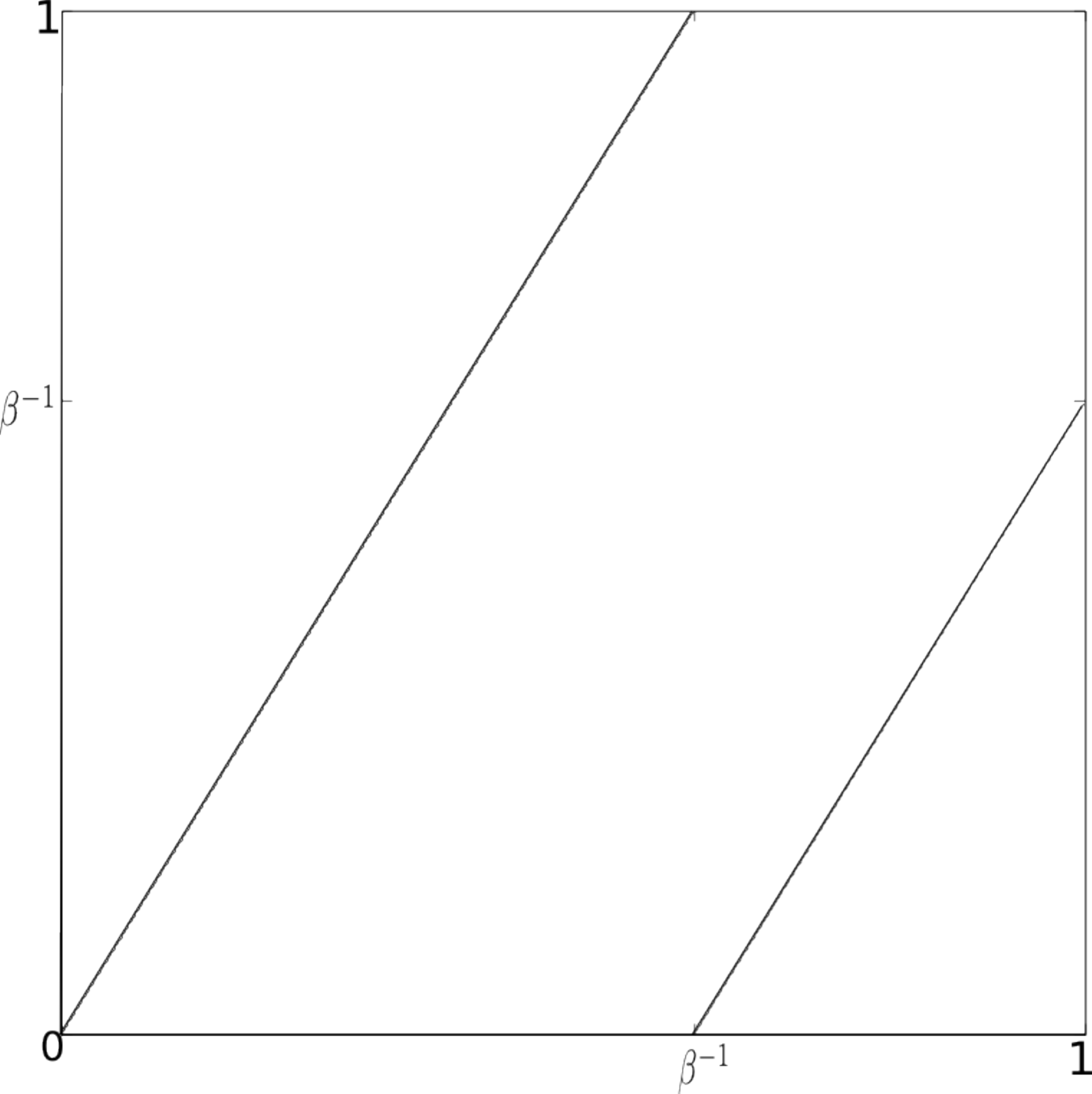}\caption{\label{fig:-betatransform}The graph of the $\beta$-transform.}

\end{figure}

From \cite{Re57,Pa60} we know that there exists an invariant measure
$\nu$ absolutely continuous to the Lebesgue measure restricted to
$[0,1]$ with density $h$ given by \[
h(x):=\begin{cases}
\frac{5+3\sqrt{5}}{10} & \text{for}\,\,0\leq x<\frac{\sqrt{5}-1}{2},\\
\frac{5+\sqrt{5}}{10} & \text{for}\,\,\frac{\sqrt{5}-1}{2}\leq x<1.\end{cases}\]
The measure $\nu$ can be represented on $\Sigma_{A}$ by a stationary
Markov measure with the stochastic matrix \[
\Pi:=\left(\begin{array}{cc}
\beta-1 & 2-\beta\\
1 & 0\end{array}\right)\]
 and probability vector $p:=\left(\frac{\beta}{\sqrt{5}},\frac{\beta-1}{\sqrt{5}}\right)$.
The scaling operator $U$ acting on $L^{2}\left(\nu_{\Z}\right)$
is then given for $x\in\R$ by

\begin{align*}
Uf(x)= & \sum_{k\in\Z}\Big(\sqrt{\beta}\mathbbm{1}_{[0,\beta^{-2})}(x-k)+\mathbbm{1}_{[\beta^{-2},\beta^{-1})}(x-k)\\
 & \,\,\,\,\,\,\,\,+\beta\cdot\mathbbm{1}_{[\beta^{-1},1)}(x-k)\Big)\cdot f\left(\beta(x-k)+2k\right).\end{align*}
 For the father wavelets we may choose $\varphi_{0}=\left(\sqrt{5}/\beta\right)^{1/2}\mathbbm{1}_{[0,\beta-1)}$
and $\varphi_{1}=\left(\sqrt{5}\beta\right)^{1/2}\mathbbm{1}_{[\beta-1,1)}$.
We illustrate the action of $U$ in Figure \ref{fig:Application-of}
where $U$ is applied to the identity map $\id_{\left[0,1\right]}:x\mapsto x$,
restricted to $\left[0,1\right]$, that is for $x\in[0,1]$ we have
\[
U\left(\id_{\left[0,1\right]}\right)x=\left(\sqrt{\beta}\mathbbm{1}_{[0,\beta^{-2})}(x)+\mathbbm{1}_{[\beta^{-2},\beta^{-1})}(x)+\beta\cdot\mathbbm{1}_{[\beta^{-1},1)}(x)\right)\beta x.\]

\end{example}
\begin{figure}
\includegraphics[width=0.6\textwidth]{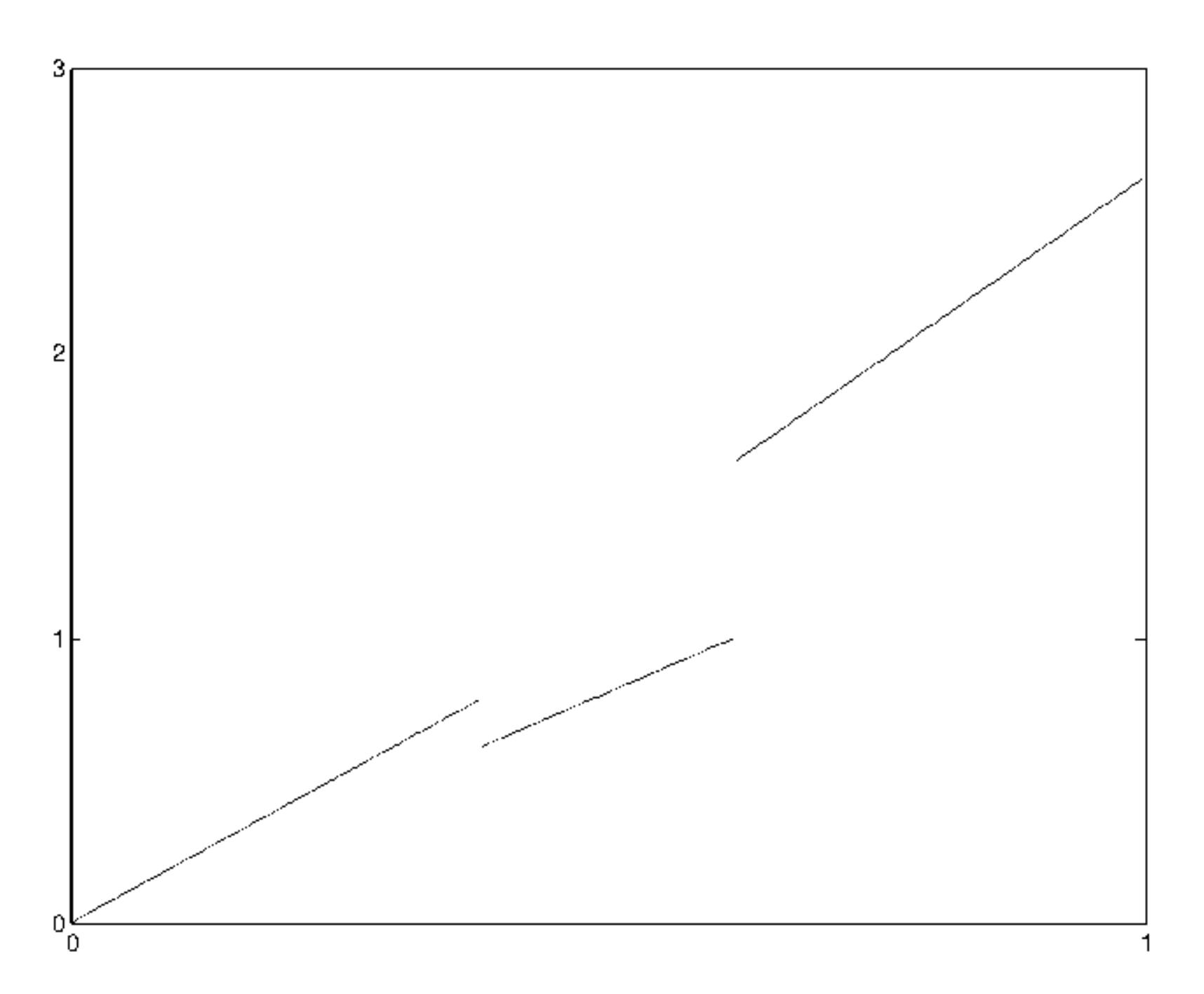}

\caption{\label{fig:Application-of}The graph of $U\left(\id_{[0,1]}\right)$.}

\end{figure}

We further generalize our construction by considering non-atomic probability
measures $\nu$ on $X$ which we do not assume to be Markovian. In
this case it is natural to consider more than just one scaling operator
$U$. More precisely, we consider a family of scaling operators $\left(U^{(n)}\right)_{n\in\Z}$
which allow us to construct an orthonormal wavelet basis. For this
we define $U^{(0)}:=\id$ and for $f\in L^{2}(\nu_{\Z})$ and $n\in\N$
we let \begin{equation}
\begin{array}{ll}
U^{(n)}f(x) & :={\displaystyle \sum_{k\in\Z}\sum_{\omega\in\Sigma_{A}^{n}}\sum_{j\in\underline{N}}\sqrt{\frac{\nu_{\Z}([j])}{\nu_{\Z}([\omega j])}}\1_{[\omega j]}(x-k)}\\
 & \qquad\qquad\qquad\;\cdot f\left(\tau_{\omega}^{-1}(x-k)+\sum_{i=0}^{n-1}\omega_{n-1-i}N^{i}+N^{n}k\right)\end{array}\label{eq:Def U^n}\end{equation}

and\begin{eqnarray}
U^{(-n)}f(x) & := & \sum_{a}\sum_{k\in\Z}\sum_{\omega\in\Sigma_{A}^{n}}\sum_{j\in\underline{N}}\sqrt{\frac{\nu_{\Z}([\omega j])}{\nu_{\Z}([j])}}\1_{[j]}\left(x-\sum_{i=0}^{n-1}\omega_{n-1-i}N^{i}-N^{n}k\right)\nonumber \\
 &  & \ \qquad\qquad\ \qquad\qquad\cdot f\left(\tau_{\omega}\left(x-\sum_{i=0}^{n-1}\omega_{n-1-i}N^{i}-N^{n}k\right)+k\right).\label{eq:Def U^-n}\end{eqnarray}
 It is straight forward to verify that if the measure $\nu$ is Markovian,
then we have $U^{(n)}=U^{n}$ for $n\in\N_{0}$ and $U^{(-n)}=\left(U^{*}\right)^{n}$,
$n\in\N$. More details are provided in Section \ref{sub:MRA-for-Markov}.
Furthermore, the operators $\left(U^{(n)}\right)_{n\in\Z}$ and $T$
satisfy the following relations.
\begin{prop}
\label{pro:The-operatorsUT-1} Let $\left(\varphi_{j}\right)_{j\in\underline{N}}$
denote the family of father wavelets given by $\varphi_{i}=\left(\nu_{\Z}([i])\right)^{-1/2}\mathbbm{1}_{[i]}$,
$i\in\underline{N}$. The translation operator $T$ and the family
of scaling operators $\left(U^{(n)}\right)_{n\in\Z}$ satisfy the
following.
\begin{enumerate}
\item \label{enu:(1)TU-1} $TU^{(n)}=U^{(n)}T^{N^{n}}$, $n\in\N$,
\item \label{enu:(2)Znt-1}$U^{(-n)}T\varphi_{j}=T^{N^{n}}U^{(-n)}\varphi_{j}$,
$n\in\N$, $j\in\underline{N}$,
\item \label{enu:(3)TU-1-1} $\varphi_{i}=U^{(1)}T^{i}\sum_{j\in\underline{N}}\sqrt{\frac{\nu_{\Z}([ij])}{\nu_{\Z}([i])}}\varphi_{j}$,
$i\in\underline{N}$,
\item \label{enu:(4)TU-1} $\langle U^{(n)}T^{k}\varphi_{i}|U^{(n)}T^{l}\varphi_{j}\rangle=\delta_{(k,i),(l,j)}$,
$n,k,l\in\Z$, $i,j\in\underline{N}$,
\item \label{enu:(5)UZ-1}$U^{(n)}U^{(-n)}=\id$, $n\in\N$,
\item \label{enu:(6)ZU-1}if $U^{(n)}T^{k}\varphi_{j}\neq0$, then $U^{(-n)}U^{(n)}T^{k}\varphi_{j}=T^{k}\varphi_{j}$,
$n\in\N$, $k\in\Z$, $j\in\underline{N}$. 
\end{enumerate}
\end{prop}
The properties of $\left(U^{(n)}\right)_{n\in\Z}$ and $T$ lead us
to the following abstract definition of a multiresolution analysis
which involves more than one father wavelet. In the literature these
functions are sometimes called multiwavelets (cf. \cite{Al93}).
\begin{defn}
[Abstract MRA]\label{def:MRA allgemein}Let $\mu$ be a non-atomic
measure on $\left(\R,\mathcal{B}\right)$.
\begin{enumerate}
\item Let $\left(\U^{(n)}\right)_{n\in\Z}$ and $\T$ be bounded, linear
operators on $L^{2}(\mu)$ such that $\T$ is unitary and $\U^{(0)}=\id$.
We say $\left(\mu,\left(\U^{(n)}\right)_{n\in\Z},\T\right)$ allows
a\textit{\emph{ }}\emph{two-sided multiresolution analysis (two-sided
MRA)} if there exists a family $\left\{ V_{j}:\, j\in\Z\right\} $
of closed subspaces of $L^{2}\left(\mu\right)$ and a family of functions
(called \emph{father wavelets}) $\varphi_{j}\in L^{2}\left(\mu\right)$,
$j\in\underline{N}$, $N\in\N$, with compact support, such that the
following conditions are satisfied. 

\begin{enumerate}
\item $\dots\subset V_{-2}\subset V_{-1}\subset V_{0}\subset V_{1}\subset V_{2}\subset\dots$,
\label{enu:DefMRA1-1}
\item $\text{cl}\bigcup_{j\in\Z}V_{j}=L^{2}(\mu)$,\label{enu:DefMRA2-1}
\item $\bigcap_{j\in\mathbb{Z}}V_{j}=\{0\}$, \label{enu:DefMRA3-1}
\item $\left(\U^{(n)}\left\{ \T^{k}\varphi_{j}:\; k\in\mathbb{Z},\, j\in\underline{N}\right\} \right)\backslash\left\{ 0\right\} $,
$n\in\Z$, is an orthonormal basis of $V_{n}$,\label{enu:DefMRA4-1}
\item $\U^{(n)}\left\{ \T^{k}\varphi_{i}:k\in\underline{N^{n}},i\in\underline{N}\right\} \subset\spn\U^{(n+1)}\left\{ \T^{k}\varphi_{i}:k\in\underline{N^{n+1}},i\in\underline{N}\right\} $,
$n\in\N_{0}$ and $\U^{(-n)}\left\{ \varphi_{i}:i\in\underline{N}\right\} \subset\spn\U^{(-n+1)}\left\{ \T^{k}\varphi_{i}:i\in\underline{N},k\in\underline{N}\right\} $,\label{enu:-DefMRA5-1}
$n\in\N$,
\item $\T\U^{(n)}|_{V_{0}}=\U^{(n)}\T^{N^{n}}|_{V_{0}}$ and $\U^{(-n)}\T|_{V_{0}}=\T^{N^{n}}\U^{(-n)}|_{V_{0}}$,
$n\in\N$.\label{enu:DefMRA6-1}
\end{enumerate}
\item Let $\left(\U^{(n)}\right)_{n\in\N_{0}}$ and $\T$ be linear operators
on $L^{2}(\mu)$, $\T$ unitary, and let $\U^{(0)}=\id$. We say $\left(\mu,\left(\U^{(n)}\right)_{n\in\N_{0}},\T\right)$
allows a \textit{one-sided}\textit{\emph{ }}\emph{multiresolution
analysis (}\textit{one-sided}\emph{ MRA)} if there exists a family
$\left(V_{j}:\, j\in\N_{0}\right)$ of closed subspaces of $L^{2}\left(\mu\right)$
and a family of functions (called \emph{father wavelets}) $\varphi_{j}\in L^{2}\left(\mu\right)$,
$j\in\underline{N}$, $N\in\N$, with compact support, such that the
following conditions are satisfied. 

\begin{enumerate}
\item $V_{0}\subset V_{1}\subset V_{2}\subset\dots$, \label{enu:DefMRA2-1-1}
\item $\text{cl}\bigcup_{j\in\N_{0}}V_{j}=L^{2}(\mu)$,\label{enu:DefMRA2-2-1}
\item $\left(\U^{(n)}\left\{ \T^{k}\varphi_{j}:\; k\in\mathbb{Z},\, j\in\underline{N}\right\} \right)\backslash\left\{ 0\right\} $,
$n\in\N_{0}$, is an orthonormal basis of $V_{n}$,\label{enu:DefMRA3-2} 
\item $\U^{(n)}\left\{ \T^{k}\varphi_{i}:k\in\underline{N^{n}},i\in\underline{N}\right\} \subset\spn\U^{(n+1)}\left\{ \T^{k}\varphi_{i}:k\in\underline{N^{n+1}},i\in\underline{N}\right\} $,
$n\in\N$,\label{enu:DefMRA4-2}
\item $\T\U^{(n)}|_{V_{0}}=\U^{(n)}\T^{N^{n}}|_{V_{0}}$, $n\in\N$.\label{enu:DefMRA5-2}
\end{enumerate}
\end{enumerate}
\end{defn}
Our next theorem shows that the abstract MRA holds in particular for
MIM as introduced above.
\begin{thm}
\label{thm:MRa}Let $\left(U^{(n)}\right)_{n\in\N_{0}}$ be given
as in (\ref{eq:Def U^n}). Then $\left(\nu_{\Z},\left(U^{(n)}\right)_{n\in\N_{0}},T\right)$
allows a one-sided MRA, where the father wavelets are set to be $\varphi_{i}:=\left(\nu_{\Z}([i])\right)^{-1/2}\mathbbm{1}_{[i]}$,
$i\in\underline{N}$. 
\end{thm}
For the abstract MRA we show that there always exists an orthonormal
wavelet basis. 
\begin{thm}
\label{thm:MRA}Let $\mu$ be a non-atomic measure on $\R$, $\left(\U^{(n)}\right)_{n\in\Z}$
be a family of bounded, linear operators on $L^{2}(\mu)$ and $\T$
be a unitary operator on $L^{2}(\mu)$. If $\left(\mu,\left(\U^{(n)}\right)_{n\in\Z},\T\right)$
allows a two-sided MRA with father wavelets $\varphi_{j}$, $j\in\underline{N}$,
then there exist for every $n\in\N_{0}$ numbers $d_{n}\in\underline{N^{n+2}}$,
$d_{-n}\in\underline{N^{2}},$ $q_{n}\in\underline{N^{n+1}}$, $q_{-n}\in\underline{N}$,
with $d_{n}\geq q_{n}$, $d_{-n}\geq q_{-n}$, and two families of
mother wavelets $\left(\psi_{n,l}:l\in\underline{d_{n}-q_{n}}\right)$,
$\left(\psi_{-n,l}:l\in\underline{d_{-n}-q_{-n}}\right)$, $n\in\N_{0}$,
such that the following set of functions defines an orthonormal basis
for $L^{2}(\mu)$ \begin{eqnarray*}
 &  & \left\{ \T^{k}\psi_{n,l}:\, n\in\N_{0},\, l\in\underline{d_{n}-q_{n}},\, k\in\Z\right\} \\
 &  & \qquad\qquad\qquad\qquad\qquad\qquad\cup\left\{ \T^{N^{n}k}\psi_{-n,l}:\, n\in\N,\, l\in\underline{d_{-n}-q{}_{-n}},\, k\in\Z\right\} .\end{eqnarray*}
\end{thm}
\begin{rem}
We give a precise construction for the family of mother wavelets $\psi_{n,l}$
in Section \ref{sub:Proof-of-Theorem}. More precisely, for each $n\in\Z$
we consider the linear subspaces $W_{n}:=V_{n+1}\ominus V_{n}$, where
the closed subspaces $V_{n}$ of $L^{2}(\mu)$ are given in Definition
\ref{def:MRA allgemein} (\ref{enu:DefMRA4-1}), and the finite family
of functions $\left(\psi_{n,l}:l\in\underline{d_{n}-q_{n}}\right)$
and show that that for $n\geq0$ $\left\{ \T^{k}\psi_{n,l}:k\in\Z,l\in\underline{d_{n}-q_{n}}\right\} $,
and respectively for $n<0$ $\left\{ \T^{N^{|n|}k}\psi_{n,l}:k\in\Z,l\in\underline{d{}_{-|n|}-q{}_{-|n|}}\right\} $,
defines an orthonormal basis of $W_{n}$. 

Note that for IFS the mother wavelets are typically constructed in
terms of so-called filter functions. We will see in Section \ref{sec:Operator-algebra}
that an analog construction is still possible if the measure $\nu$
is Markovian.
\end{rem}
An immediate consequence of the proof of Theorem \ref{thm:MRA} is
the following corresponding result for the one-sided MRA.
\begin{cor}
\label{cor:onesided}Let $\mu$ be a non-atomic measure on $\R$,
$\left(\U^{(n)}\right)_{n\in\N_{0}}$ a family of bounded, linear
operators on $L^{2}(\mu)$ and $\T$ a unitary operator on $L^{2}(\mu)$.
If $\left(\mu,\left(\U^{(n)}\right)_{n\in\N_{0}},\T\right)$ allows
a one-sided MRA with the father wavelets $\varphi_{j}$, $j\in\underline{N}$,
then there exists for every $n\in\N_{0}$ numbers $d_{n}\in\underline{N^{n+2}}$,
$q_{n}\in\underline{N^{n+1}}$ with $d_{n}\geq q_{n}$ and a family
of mother wavelets $\left(\psi_{n,l}:l\in\underline{d_{n}-q_{n}}\right)$,
$n\in\N_{0}$, such that the following set of functions defines an
orthonormal basis for $L^{2}(\mu)$ \[
\left\{ \T^{k}\psi_{n,l}:\, n\in\N_{0},\, l\in\underline{d_{n}-q_{n}},\, k\in\Z\right\} \cup\left\{ \T^{k}\varphi_{i}:k\in\Z,n\in\underline{N}\right\} .\]

\end{cor}
The construction for a MIM with an underlying Markov measure $\nu$
belongs to a specific class. In this class the scaling operators $\U^{(n)}$
can be represented multiplicatively. In our general framework we say
that $\left(\mu,\left(\U^{(n)}\right)_{n\in\N_{0}},\T\right)$ is
multiplicative if there exists a linear, bounded operator $\U$ on
$L^{2}\left(\mu\right)$ such that for $n\in\N_{0}$ $\U^{(n)}=\U^{n}$
and $\U^{(-n)}=\left(\U^{*}\right)^{n}$. The results concerning the
mother wavelets simplify in this case as a consequence of the following
lemma.
\begin{lem}
\label{lem:Wn} Let us assume that $\left(\mu,\left(\U^{(n)}\right)_{n\in\N_{0}},\T\right)$
allows a two-sided MRA with the closed subspaces $V_{n}$ of $L^{2}(\mu)$
from Definition \ref{def:MRA allgemein} (\ref{enu:DefMRA4-1}) and
set $W_{n}:=V_{n+1}\ominus V_{n}$, $n\in\Z$. 
\begin{itemize}
\item If $\U^{(n)}=\U^{n}$ for $n\in\N$ then $W_{n}=\U^{n}W_{0}$, $n\in\N$. 
\item If $\U^{(-n)}=\left(\U^{*}\right)^{n}$ for $n\in\N_{0}$ then $W_{-n}=\left(\U^{*}\right)^{n-1}W_{-1}$. 
\end{itemize}
\end{lem}
Thus, we only have to find appropriate mother wavelets for $W_{0}$
and $W_{-1}$ and obtain a wavelet basis then by applying repeatedly
$\U$. More precisely, this observation allows us to derive the following
corollary from the Theorem \ref{thm:MRA}.
\begin{cor}
\label{cor:basis U}If $\left(\mu,\left(\U^{(n)}\right)_{n\in\N_{0}},\T\right)$
is multiplicative, then there exists an orthonormal basis of $L^{2}(\mu)$
of the form \begin{align*}
 & \left\{ \U^{n}\T^{k}\psi_{l}:n\in\N_{0},\omega\in\Sigma_{A}^{n},k=\sum_{i=0}^{n-1}\omega_{i}N^{i}+N^{n}m,m\in\Z,l\in\underline{d_{1}-N}\right\} \\
 & \qquad\qquad\qquad\qquad\qquad\qquad\cup\left\{ \left(\U^{*}\right)^{n}\T^{k}\psi_{-,l}:n\in\N_{0},k\in\Z,l\in\underline{d_{-1}-N}\right\} ,\end{align*}
where the functions $\psi_{l}$, $l\in\underline{d_{1}-N}$, and $\psi_{-,l}$,
$l\in\underline{d_{-1}-N}$, are given explicitly in Remark \ref{rem:motherwaveletsmulitplicative}.
\end{cor}
The above corollary applied to Example \ref{exa:-Transformation}
leads to the following construction.
\begin{example*}
[Example \ref{exa:-Transformation} (continued)]

The mother wavelet is\begin{align*}
\psi & =\left(\sqrt{5}(2-\beta)\right)^{1/2}\mathbbm{1}_{[0,(\beta-1)^{2})}-\left(\sqrt{5}\right)^{1/2}\mathbbm{1}_{[(\beta-1)^{2},\beta-1)}\end{align*}
and so a basis is given by \[
\left\{ T^{k}\varphi_{1}:\, k\in2\Z+1\right\} \cup\left\{ U^{n}T^{k}\psi:\, k\in D_{n},\, n\in\N\right\} \cup\left\{ \left(U^{*}\right)^{n}T^{k}\psi:\, k\in\Z,\, n\in\N\right\} ,\]
where \[
D_{n}:=\left\{ \sum_{j=0}^{n-1}k_{j}2^{j}+2^{n}l:\left(k_{j}\right)_{j\in\underline{n}}\in\left\{ 0,1\right\} ^{n},k_{j}\cdot k_{j-1}=0,j\in\underline{n-1},\, l\in\Z\right\} .\]
The proof that this indeed defines a orthonormal basis will be postponed
to Section \ref{sub:Examples}.
\end{example*}
In the case of a MRA for a MIM with Markov measure $\nu$ we have
in particular that $U^{(n)}=U^{n}$ and $U^{(-n)}=\left(U^{*}\right)^{n}$and
we even obtain a stronger correspondence between Markov measures for
MIM and a two-sided MRA. 
\begin{thm}
\label{thm:MRA-1} We have that $\left(\nu_{\Z},\left(U^{(n)}\right)_{n\in\Z},T\right)$
allows a two-sided MRA with the father wavelets $\varphi_{i}:=\left(\nu_{\Z}([i])\right)^{-1/2}\mathbbm{1}_{[i]}$,
$i\in\underline{N}$, if and only if the measure $\nu$ is Markovian. 
\end{thm}
In the case of $\nu$ being a Markov measure we even have a stronger
property appart from being multiplicative, that is we have $\varphi_{j}\in\spn U\left\{ T^{j}\varphi_{i}:i\in\underline{N}\right\} $
for each $j\in\underline{N}$. We call a MRA with this property \emph{translation
complete}. We further investigate multiplicative MRA which are translation
complete in Section \ref{sub:Multiresolution-analysis-for}. In this
situation we derive a 0-1-valued transition matrix $A$ given by $A_{ij}=0$
if and only if $\U\T^{i}\varphi_{j}=0$ and show that for MIM the
matix coincides with the incidence matrix. This observation is used
to construct the mother wavelets in a simpler way by considering for
each father wavelet a unitary matrix to obtain coefficients for the
corresponding mother wavelets. We will use this approach to construct
the mother wavelets for MIM. 

We would like to point out some interesting connections to $C^{*}$-algebras
of Cuntz-Krieger type, \cite{KeStaStr07}. We start by further considering
the scaling operator $U$ for the MRA in the setting of a MIM with
the incidence matrix $A$ and Markov measure $\nu$. We can also write
the operator $U$ in a different way using the representation of a
Cuntz-Krieger-algebra. For this we consider the partial isometries
$S_{i}$ given for $i\in\underline{N}$, $f\in L^{2}(\nu)$, $x\in\supp(\nu)$
by \[
S_{i}f(x)=\left(\nu([i])\right)^{-1/2}\mathbbm{1}_{[i]}(x)f(\tau_{i}^{-1}(x)).\]
It has been shown in \cite{KeStaStr07} that this gives a representation
of the Cuntz-Krieger-algebra $\mathcal{O}_{A}$ by bounded operators
acting on $L^{2}(\nu)$, that is the $S_{i}$, $i\in\underline{N}$,
are partial isometries and satisfy \begin{alignat*}{1}
S_{i}^{*}S_{i} & =\sum_{j\in\underline{N}}A_{ij}S_{j}S_{j}^{*},\\
1 & =\sum_{i\in\underline{N}}S_{i}S_{i}^{*}.\end{alignat*}
The scaling operator $U$ acting on $L^{2}\left(\nu_{\Z}\right)$
can then alternatively be written in terms of the partial isometries
as \[
U=\sum_{k\in\Z}\sum_{j\in\underline{N}}\sum_{i\in\underline{N}}\sqrt{\frac{p_{i}}{\pi_{ji}}}T^{k}S_{j}\mathbbm{1}_{[i]}T^{-(j+Nk)},\]
where we notice that $S_{j}\1_{[i]}$, $j,i\in\underline{N}$, acts
on $L^{2}\left(\nu_{\Z}\right)$. We can also write $U^{*}$ in terms
of the partial isometries $S_{i}$, $i\in\underline{N}$. In this
way we obtain \[
U^{*}=\sum_{k\in\Z}\sum_{j\in\underline{N}}\sum_{i\in\underline{N}}\sqrt{\frac{\pi_{ji}}{p_{i}}}T^{j+Nk}\mathbbm{1}_{[i]}S_{j}^{*}T^{-k}.\]
 The spaces $V_{n}$, $n\in\N_{0}$, can also be written in terms
of the isometries $S_{i}$, $i\in\underline{N}$, that is for $n\in\N$
a basis of $V_{n}$ is given by

\[
\left\{ \sqrt{\frac{p_{i}}{\pi_{\omega_{n-1}i}}}T^{l}S_{\omega}\varphi_{i}:\, l\in\Z,\omega\in\Sigma_{A}^{n},\, i\in\underline{N}\right\} .\]

Let us finish this section by commenting on some known results in
the literature connected to the results in here. Up to our knowledge
there are at least two further approaches to construct a wavelet basis
on the limit sets of MIM, namely \cite{MaPa09,KeSa10} and there is
one approach for the specific case of a $\beta$-transformation given
in \cite{GP96}. In \cite{MaPa09} Marcolli and Paolucci consider
the limit set $X$ of a MIM inside the unit interval consisting of
the inverse branches $\tau_{i}(x)=\frac{x+i}{N}$ for $i\in\underline{N}$
with some transition rule encoded in a matrix $A$. This limit set
can be associated with a Cantor set inside the unit interval. The
Cantor set is then equipped with the Hausdorff measure of the appropriate
dimension $\delta$. If all transitions were allowed, the limit set
would coincide with a usual Cantor set given by an affine iterated
function system. They then use the representation of the Cuntz-Krieger-algebra
$\mathcal{O}_{A}$, where $A$ is the transition matrix, for the construction
of the orthonormal system of wavelets on $L^{2}\left(H^{\delta}|_{X}\right)$
and not a multiresolution analysis. Their proofs mainly rely on results
in \cite{Bod06,Jo98}. Finally, Marcolli and Paolucci give a possible
application where they adapt the construction of a wavelet basis to
graph wavelets for finite graphs with no sinks, which can be associated
to Cuntz-Krieger-algebras. These graph wavelets are a useful tool
for spatial network traffic analysis, compare \cite{MaPa09,CK03}. 

In \cite{KeSa10} the authors construct a Haar basis analogous to
the wavelet basis construction in \cite{DJ03} for the middle third
Cantor set for a one-sided topologically exact subshift of finite
type and with respect to a Gibbs measure $\mu_{\phi}$ for a Hölder
continuous potential $\phi$. The construction is then used to obtain
a spectral triple in the framework of non-commutative geometry.

The construction of wavelet basis in different spaces than $L^{2}(\lambda)$,
where $\lambda$ is the Lebesgue measure on $\R$, may lead to a further
understanding of non-commutative geometry in the sense that we can
obtain a {}``Fourier'' or wavelet basis for quasi lattices or quasi
crystals. 

As an essential non-linear example for the construction of a wavelet
basis on limit sets of MIM one can take the limit set of a Kleinian
group together with the measure of maximal entropy or the Patterson-Sullivan
measure, compare Example \ref{exa:cocompact}.

As an example we apply the construction to a $\beta$-transformation,
where $\beta$ denotes the golden mean, i.e. $\beta=\frac{1+\sqrt{5}}{2}$,
compare Example \ref{exa:-Transformation}. In this way we obtain
a wavelet basis for $L^{2}\left(\nu_{\Z}\right),$ where $\nu$ is
the invariant measure for this transformation, compare \cite{Re57,Pa60}.
This measure is absolutely continuous with respect to the Lebesgue
measure. In \cite{GP96}, Gazeau and Patera construct a similar basis
to ours for the $\beta$-transformation with respect to the Lebesgue
measure on $\R$. They use instead of a translation by the group $\Z$
a translation by so called $\beta$-integers which consider the $\beta$-adic
expansion and are obtained by a so-called greedy algorithm. There
are some common features between our construction and the one in \cite{GP96}
like both give characteristic functions on intervals depending on
powers of $\beta$. But since we consider different measures we have
different coefficients. 

The paper is organized as follows. In Section \ref{sec:Setting-(IFS-with}
we provide some basic definitions and introduce MIM. In Section \ref{sec:Multiresolution-analysis}
we elaborate the abstract MRA for families of operators $\left(\U^{(n)}\right)_{n\in\Z}$
and give a proof of Theorem \ref{thm:MRA}. In Section \ref{sub:Multiresolution-analysis-for}
we then consider the special case of multiplicative systems. In Section
\ref{sub:Translation-completeness} we prove how the condition of
translation completeness simplifies the construction of the mother
wavelets. The rest of this paper is devoted to the special case of
a MRA for MIM. In Section \ref{sec:Construction-of-a} we start with
a family of operators $\left(U^{(n)}\right)_{n\in\Z}$ acting on $L^{2}\left(\nu_{\Z}\right)$
for an arbitrary non-atomic probability measure $\nu$ on the limit
set of an MIM in the unit interval translated by $\Z$ and show that
a one-sided MRA is always satisfied. If on the other hand a two-sided
MRA holds, we then prove that the measure $\nu$ is necessarily Markovian.
The construction of the mother wavelets will be given explicitly.
In Section \ref{sub:MRA-for-Markov} we give an explicit construct
of the wavelet basis if the measure $\nu$ is Markovian. 

Finally, in Section \ref{sec:Operator-algebra} we show how low-pass
filters and high-pass filters can be employed to construct mother
wavelets for multiwavelets for MIM with an underlying Markov measure.

\section{\label{sec:Setting-(IFS-with}Markov Interval Maps}

In this section we give some basic definitions and notations. We consider
fractals given as limit sets of one-dimensional Markov Interval Maps. 
\begin{defn}
\label{def:MFS}Let $\left(B_{i}\right)_{i\in\underline{N}}$ be closed
intervals in $[0,1]$ with disjoint interior. Define $I:=\bigcup_{i\in\underline{N}}B_{i}$
and $F:I\to[0,1]$ exanding and $C^{1}$ on each $B_{i}$, $i\in\underline{N}$,
such that if $F(B_{i})\cap B_{j}\neq\emptyset$ then $B_{j}\subset F(B_{i})$
for $i,j\in\underline{N}$. We call the system $\left(\left(B_{i}\right)_{i\in\underline{N}},F\right)$
a Markov Interval Map and its limit set is defined as $X:=\bigcap_{n=0}^{\infty}F^{-n}I$.\end{defn}
\begin{rem}
$\ $
\begin{enumerate}
\item If $F(B_{i})=[0,1]$ for each $i\in\underline{N}$, then $\left(X,F\right)$
corresponds to an iterated function system (IFS). 
\item We define the inverse branches $\tau_{i}:=\left(F|_{B_{i}}\right)^{-1}$,
$i\in\underline{N}$. The family $\left(\tau_{i}\right)_{i\in\underline{N}}$
is called a one-dimensional graph directed Markov system (GDMS) with
the incidence matrix $A=\left(A_{ij}\right)_{i,j=0}^{N-1}$ which
is obtained by \[
A_{ij}:=\begin{cases}
1, & \text{if}\,\, B_{j}\subset F\left(B_{i}\right)\\
0, & \text{else},\end{cases}\]
 and it follows that $F\left(B_{i}\right)=\bigcup_{j\in\underline{N}:\, A_{ij}=1}B_{j}$. 
\end{enumerate}
\end{rem}
\begin{example}
\label{exa:cocompact}An example is a convex, co-compact Kleinian
group, as an example consider Figure \ref{fig:In-hyperbolic-space}.
The limit set can be considered as the limit set of the Bowen-Series
map, which gives rise to a Markov Interval Map, compare Figure \ref{fig:Bowen-Series-map}.
The limit set is the set that is obtained by successive application
of these four maps, where the composition of $g_{i}$ and $g_{i}^{-1}$are
forbidden. A typical measure to be studied would be the measure of
maximal entropy or the conformal measure (of maximal dimension).
\end{example}
\begin{figure}
\subfloat[\label{fig:In-hyperbolic-space}A fundamental domain of the action
of $\langle g,h\rangle$ on the Poincaré disc model.]{\includegraphics[width=0.47\textwidth]{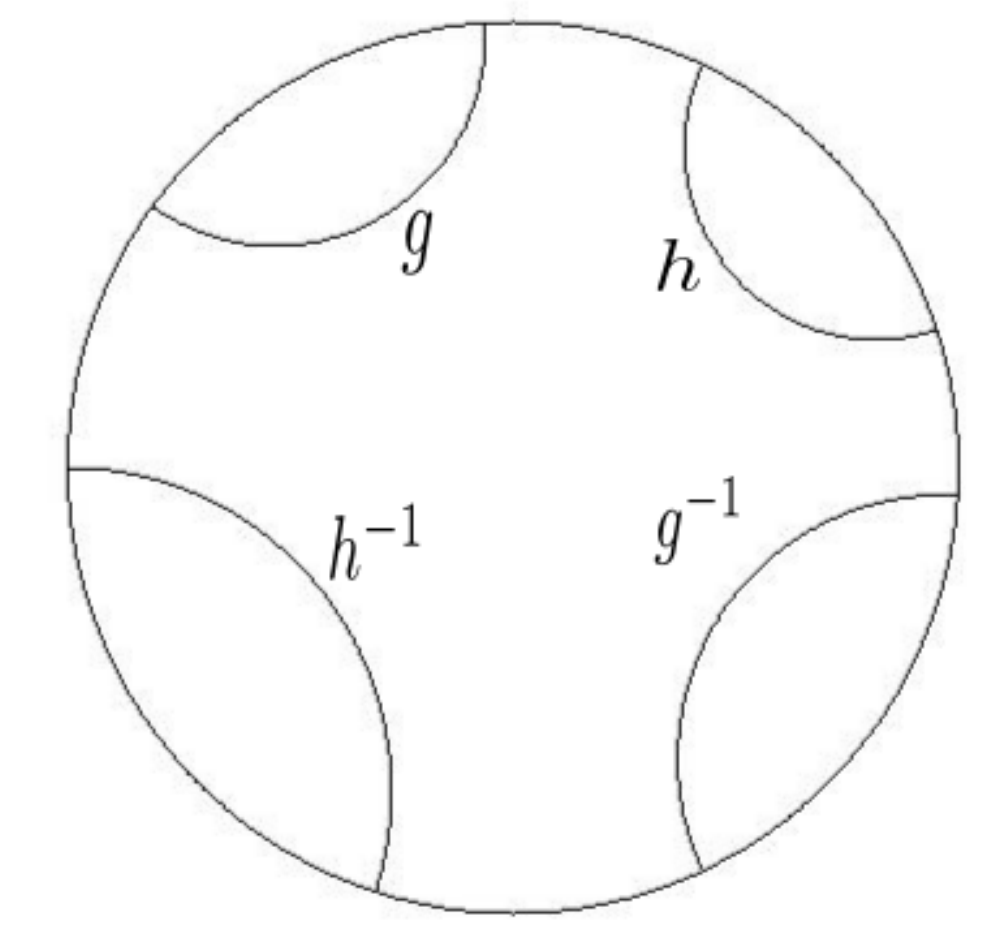}}\quad{} \subfloat[\label{fig:Bowen-Series-map}The corresponding Bowen-Series map.]{\includegraphics[width=0.42\textwidth]{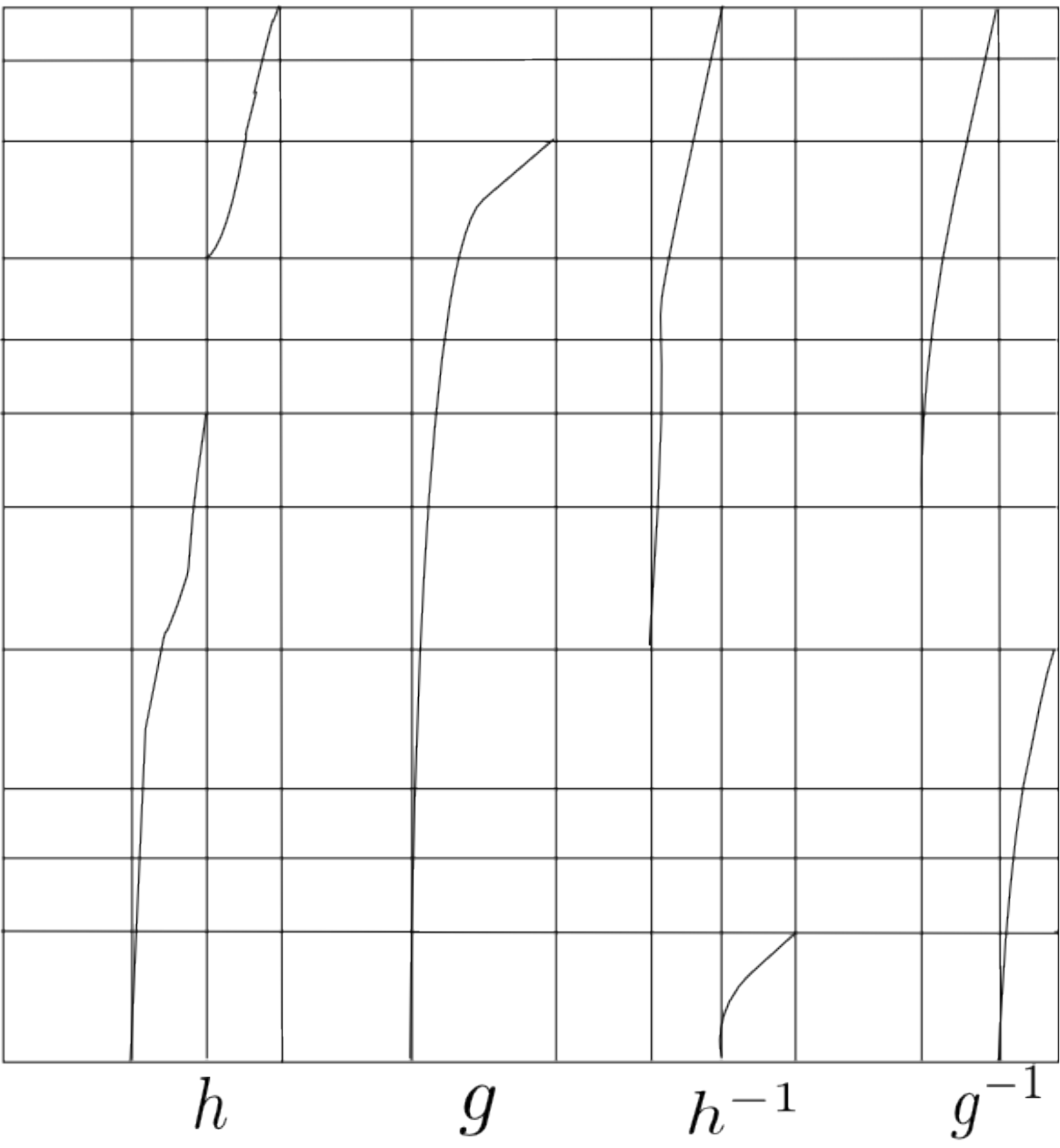}}\caption{Example of a Fuchsian group.}

\end{figure}

Next we consider the corresponding shift space. Consider the alphabet
$\underline{N}=\left\{ 0,\dots,N-1\right\} $. The limit set $X$
is then homeomorphic (mod $\nu$) to the set of all admissible words
\[
\Sigma_{A}:=\{\omega=(\omega_{0},\omega_{1},\dots)\in\underline{N}^{\N}:\, A_{\omega_{i}\omega_{i+1}}=1\,\forall i\geq0\}.\]
The homeomorphism is given, for $x\in X$, by\begin{equation}
\begin{array}{ll}
\pi: & \Sigma_{A}\rightarrow X\\
 & \omega\mapsto\lim_{n\rightarrow\infty}\tau_{\omega_{0}}\circ\dots\circ\tau_{\omega_{n}}(x),\end{array}\label{eq:codingmap}\end{equation}
which is independent of the particular choice of $x\in X$. 
\begin{rem}
Furthermore, we define the cylinder sets for $\omega_{0},\dots,\omega_{k}\in\underline{N}$,
$k\in\N_{0}$, as \[
\left[\omega_{0}\dots\omega_{k}\right]:=\left\{ \left(\omega_{0}',\omega_{1}',\dots\right)\in\Sigma_{A}:\,\omega_{i}=\omega_{i}',\, i\in\{0,\dots,k\}\right\} .\]
If for some $i\in\{0,\dots,k-1\}$ $A_{\omega_{i}\omega_{i+1}}=0$
then $\left[\omega_{0}\dots\omega_{k}\right]=\emptyset$. 

Then the sets $B_{i}$ and $F\left(B_{i}\right)$ for $i\in\underline{N}$
are homeomorphic (mod $\nu$) to the following sets in the shift space:
\[
\pi^{-1}\left(B_{i}\right)=[i]\]
and \[
\pi^{-1}\left(F(B_{i})\right)=\{\omega=\left(\omega_{0},\omega_{1},\dots\right)\in\Sigma_{A}:\, A_{i\omega_{0}}=1\}.\]
The dynamic of $F$ is conjugated to the shift dynamic $\sigma:\,\Sigma_{A}\rightarrow\Sigma_{A}$,
$\sigma\left(\omega_{0},\omega_{1},\dots\right)=\left(\omega_{1},\omega_{2},\dots\right)$
and consequently, the functions $\tau_{i}$ correspond to the inverse
branches of the shift function, i.e. $\tau_{i}\circ\pi\left(\omega_{0},\omega_{1},\dots\right)=\pi\left(i,\omega_{0},\omega_{1},\dots\right)$,
for $\omega\in\pi^{-1}\left(F(B_{i})\right)$, $i\in\underline{N}$.

Furthermore, let us fix the following notation.\end{rem}
\begin{itemize}
\item $\Sigma_{A}^{n}:=\left\{ \omega=\left(\omega_{0},\dots,\omega_{n-1}\right)\in\underline{N}^{n}:\text{ }A_{\omega_{i}\omega_{i+1}}=1\text{ for all }i\in\{0,\dots,n-1\}\right\} $
defines the set of admissible words of length $n\in\N$. 
\item $\Sigma_{A}^{*}$ stands for all finite words, i.e. $\Sigma_{A}^{*}=\bigcup_{n\geq1}\Sigma_{A}^{n}$. 
\item For $\omega\in\Sigma_{A}^{n}$ we define $\tau_{\omega}:=\tau_{\omega_{0}}\circ\tau_{\omega_{1}}\circ\dots\circ\tau_{\omega_{n-1}}$. 
\item For $\omega\in\Sigma_{A}^{n}$, $\tau\in\Sigma_{A}^{m}$ we define
their concatenation \[
\omega\tau:=\left(\omega_{0},\dots,\omega_{n-1},\tau_{0},\dots,\tau_{m-1}\right)\]
 which is an element of $\Sigma_{A}^{n+m}$ whenever $A_{\omega_{n-1}\tau_{0}}=1$. 
\end{itemize}
As a measure on $X$ we could consider for instance the pull-back
under $\pi$ of Gibbs measures on $\Sigma_{A}$ (for definitions see
e.g. \cite{KeSa10}).

Now we define the appropriate space for which we want to construct
a wavelet basis. 
\begin{defn}
\label{def:measures}Let $\widetilde{\nu}$ be a probability measure
on $\left(\Sigma_{A},\mathcal{B}\right)$ and $\nu=\widetilde{\nu}\circ\pi^{-1}$.
Define the \textit{enlarged fractal} by 

\[
R=\bigcup_{k\in\Z}X+k\]
 and define the $\Z$-convolution $\nu_{\Z}$ of the measure $\nu$
for a Borel set $B$ by \[
\nu_{\Z}(B)=\sum_{k\in\Z}\nu(B-k),\]
which clearly is an invariant measure under $\Z$-translation.\end{defn}
\begin{rem}
One example is the space $L^{2}\left(\Sigma_{A},\mu_{\phi}\right)$,
where $\Sigma_{A}$ denotes a one-sided topologically exact subshift
of finite type. An important class of measures on $\Sigma_{A}$ are
given by invariant Gibbs measure with respect to a Hölder continuous
potential $\phi\in C\left(\Sigma_{A},\R\right)$, denoted by $\mu_{\phi}$,
compare \cite{KeSa10}. $\mu_{\phi}$ corresponds to the measure $\widetilde{\nu}$
in Definition \ref{def:measures}.
\end{rem}
In the following we use the convention $0^{-1}\cdot\1_{\emptyset}=0$.
For simplicity we let $\left[\omega_{0},\dots,\omega_{n-1}\right]$
also denote the sets $\tau_{\omega_{0}}\circ\dots\circ\tau_{\omega_{n-1}}(X)$
using the identification by $\pi$. Furthermore, in Section \ref{sec:Construction-of-a}
the measure $\nu$ supported on $[0,1]$ always corresponds to a measure
$\widetilde{\nu}$ on $\Sigma_{A}$ by $\nu=\widetilde{\nu}\circ\pi^{-1}$
and $\nu_{\Z}$ denotes the measure obtained from $\nu$ by $\Z$-convolution.

\section{\label{sec:Multiresolution-analysis}Abstract Multiresolution analysis}

In this section we give a proof of Theorem \ref{thm:MRA}. To do so
we first construct mother wavelets explicitly and in the next step
we prove that these give indeed an orthonormal basis. In this section
we fix $\left(\mu,\left(\U^{(n)}\right)_{n\in\Z},\T\right)$ which
allows a two-sided MRA. 

For the construction of an ONB we cannot define the mother wavelets
in terms of filter functions due to the fact that we have more than
one father wavelet. 
\begin{rem}
\label{rem:mother wavelets}If we have the usual setting from the
literature, compare e.g. \cite{Dau92}, then we have a multiplicative
MRA with a unitary operator $\U$ and the operator $T$ given in (\ref{eq:Def T}).
In this case there is only one father wavelet $\varphi$ and there
exists a so-called low-pass filter $m_{0}:\,\mathbb{T}\rightarrow\mathbb{T}$
of the form $m_{0}(z)=\sum_{k\in\Z}a_{k}z^{k}$, $a_{k}\in\C$, such
that $\varphi=\U m_{0}(T)\varphi$. For the construction of the mother
wavelets we look for $N-1$ high-pass filters $m_{j}:\mathbb{T}\rightarrow\mathbb{T}$
of the form $m_{j}:z\mapsto\sum_{k\in\Z}b_{k}^{j}z^{k}$, $b_{k}^{j}\in\C$,
$j\in\underline{N}\backslash\{0\}$, where $N\in\N$ is connected
to the scaling since it indicates on which interval $[0,N]$ the unit
interval is mapped when the operator $\U$ is applied to $\mathbbm{1}_{[0,1]}$.
The high-pass filters are chosen, such that the matrix \[
M(z):=\frac{1}{\sqrt{N}}\left(m_{j}(\rho^{l}z)\right)_{j,l\in\underline{N}},\]
 where $\rho=e^{2\pi i/N}$, is unitary for almost all $z\in\mathbb{T}$.
In terms of these high-pass filters the mother wavelets are defined
as $\psi_{j}=\U m_{j}(T)\varphi$, $j\in\underline{N}\backslash\{0\}$.
\end{rem}
We first notice that for $n\in\N$ \begin{equation}
\left\{ \left(l,j\right)\in\underline{N^{n}}\times\underline{N}\right\} =\left\{ \left(\Big\lfloor\frac{k}{N}\Big\rfloor,(k)_{N}\right):k\in\underline{N^{n+1}}\right\} ,\label{eq:set(l,j)}\end{equation}
where $(m)_{N}:=m\mod N$ and $\lfloor x\rfloor=\max_{k\in\Z,k\leq x}(k)$
is the largest integer not exceeding $x$.

Clearly from the definition of the MRA, Definition \ref{def:MRA allgemein}
(\ref{enu:-DefMRA5-1}), we have the following:
\begin{enumerate}
\item If for $n\in\N_{0}$, $k\in\underline{N^{n+1}}$, $\U^{(n)}\T^{\lfloor\frac{k}{N}\rfloor}\varphi_{(k)_{N}}\neq0$
, there exists uniquely determined $\left(a_{m}^{n,k}\right)_{m\in\underline{N^{n+2}}}\in\C^{N^{n+2}}$
such that \begin{equation}
\begin{array}{cc}
 & \U^{(n)}\T^{\lfloor\frac{k}{N}\rfloor}\varphi_{(k)_{N}}=\U^{(n+1)}\sum_{m\in\underline{N^{n+2}}}a_{m}^{n,k}\T^{\lfloor\frac{m}{N}\rfloor}\varphi_{(m)_{N}}\\
\text{and }\\
 & \left(a_{m}^{n,k}=0,\ m\in\underline{N^{n+2}},\ \text{ if }\U^{(n+1)}\T^{\lfloor\frac{m}{N}\rfloor}\varphi_{(m)_{N}}=0\right).\end{array}\label{eq:onesided-1}\end{equation}

\item If $\U^{(-n)}\varphi_{i}\neq0$, $n\in\N$, $i\in\underline{N}$,
there exists uniquely determined coefficients $\left(b_{m}^{n,i}\right)_{m\in\underline{N^{2}}}\in\C^{N^{2}}$
such that\begin{equation}
\begin{array}{cc}
 & \U^{(-n)}\varphi_{i}=\U^{(-n+1)}\sum_{m\in\underline{N^{2}}}b_{m}^{n,i}\T^{\lfloor\frac{m}{N}\rfloor}\varphi_{(m)_{N}}\\
\text{ and }\\
 & \left(b_{m}^{n,i}=0,\ m\in\underline{N^{2}},\ \text{ if }\U^{(-n+1)}\T^{\lfloor\frac{m}{N}\rfloor}\varphi_{(m)_{N}}=0\right).\end{array}\end{equation}
 \end{enumerate}
\begin{rem}
We only consider $\U^{-(n)}\varphi_{i}$, since $\U^{(-n)}\T^{k}\varphi_{i}=\T^{N^{n}k}\U^{(-n)}\varphi_{i}$
by (\ref{enu:DefMRA6-1}) of Definition \ref{def:MRA allgemein}.\end{rem}
\begin{lem}
The following holds for the coefficients $\left(a_{m}^{n,k}\right)_{m\in\underline{N^{n+2}}}$,
$k\in\underline{N^{n+1}}$, $n\in\N_{0}$, and $\left(b_{m}^{n,i}\right)_{m\in\underline{N^{2}}}$,
$i\in\underline{N}$, $n\in\N$.
\begin{enumerate}
\item For fixed $n\in\N_{0}$, define\begin{equation}
Q_{n}:=\left\{ m\in\underline{N^{n+1}}:\U^{(n)}\T^{\lfloor\frac{m}{N}\rfloor}\varphi_{(m)_{N}}\neq0\right\} .\label{eq:def Q_n}\end{equation}
Then the vectors $v_{k}=\left(a_{m}^{n,k}\right)_{m\in\underline{N^{n+2}}}$,
$k\in Q_{n}$, are orthonormal. 
\item For fixed $n\in\N$, define \begin{equation}
Q_{-n}:=\left\{ m\in\underline{N}:\U^{(-n)}\varphi_{m}\neq0\right\} .\end{equation}
Then the vectors $w_{i}=\left(b_{m}^{n,i}\right)_{m\in\underline{N^{2}}}$,
$i\in Q_{-n}$, are orthonormal.
\end{enumerate}
\end{lem}
\begin{proof}
ad (1): For fixed $n\in\N_{0}$, let $k,l\in Q_{n}$, then \begin{align*}
\delta_{(k,l)} & =\langle\U^{(n)}\T^{\lfloor\frac{k}{N}\rfloor}\varphi_{(k)_{N}}|\U^{(n)}\T^{\lfloor\frac{l}{N}\rfloor}\varphi_{(l)_{N}}\rangle\\
 & =\big\langle\U^{(n+1)}\sum_{m\in\underline{N^{n+2}}}a_{m}^{n,k}\T^{\lfloor\frac{m}{N}\rfloor}\varphi_{(m)_{N}}\big|\U^{(n+1)}\sum_{m\in\underline{N^{n+2}}}a_{m}^{n,l}\T^{\lfloor\frac{m}{N}\rfloor}\varphi_{(m)_{N}}\big\rangle\\
 & =\sum_{m\in\underline{N^{n+2}}}a_{m}^{n,k}\overline{a}_{m}^{n,l}.\end{align*}

ad (2): Follows analogously to (1). 
\end{proof}

\subsection{\label{sub:Proof-of-Theorem}Proof of Theorem \ref{thm:MRA}}

The aim is to prove the existence of a basis as given in Theorem \ref{thm:MRA}.
For this we divide the proof in two parts. First we construct coefficients
such that the functions $\psi_{n,k}$ given in (\ref{eq:def psi1})
and (\ref{eq:def psi 2}) give an orthonormal basis. In the second
part we verify that these functions give indeed an orthonormal basis.
We prove these parts first for $n\in\N_{0}$ and then for $n\in\Z$,
$n<0$. We define the mother wavelets for each scale $n\in\Z$ such
that we obtain with their translates a basis for $W_{n}=V_{n+1}\ominus V_{n}$,
where $V_{n}$ is given in Definition \ref{def:MRA allgemein} (\ref{enu:DefMRA4-1}).
Define for $n\in\N_{0}$\begin{align*}
D_{n} & :=\left\{ m\in\underline{N^{n+2}}:a_{m}^{n,k}\neq0\text{ for some }k\in Q_{n}\right\} ,\\
D_{-n} & :=\left\{ m\in\underline{N^{2}}:b_{m}^{n,k}\neq0\text{ for some }k\in Q_{-n}\right\} ,\end{align*}
 and $d_{n}:=\card D_{n}$, $d_{-n}:=\card D_{-n}$.
\begin{itemize}
\item The mother wavelets\textit{ }for the subspaces $W_{n}$, $n\in\N_{0}$,
of the MRA shall have the form for $k\in\underline{d_{n}-q_{n}}$
with $q_{n}:=\card Q_{n}$ \begin{equation}
\psi_{n,k}:=\U^{(n+1)}\sum_{m\in\underline{N^{n+2}}}c_{m}^{n,k}\T^{\lfloor\frac{m}{N}\rfloor}\varphi_{(m)_{N}},\label{eq:def psi1}\end{equation}
where the coefficients $c_{m}^{n,k}\in\C$ are given in\textit{ }(\ref{eq:matrixone-1}).
\item For the negative index subspaces $W_{-n}$, $n\in\N$, of $L^{2}(\mu)$
we define the mother wavelets in terms of the coefficients of the
matrix in (\ref{eq:matrixtwo-1}) for $n\in\N$ and $k\in\underline{d_{-n}-q_{-n}}$,
where $q_{-n}:=\card Q_{-n}$, as \begin{equation}
\psi_{-n,k}:=\U^{(-n+1)}\sum_{m\in\underline{N^{2}}}c_{m}^{-n,k}\T^{\lfloor\frac{m}{N}\rfloor}\varphi_{(m)_{N}}.\label{eq:def psi 2}\end{equation}

\end{itemize}
The coefficients $c_{m}^{n,k}\in\C$, $c_{m}^{-n,k}\in\C$ are determined
in the following via the Gram-Schmidt process. 

For the definition of the basis we fix $n\in\N_{0}$ and we construct
an orthonormal basis for $\C^{d_{n}}$ in the following way. Consider
$\left(a_{m}^{n,k}\right)_{k\in Q_{n},m\in D_{n}}$. This is a $\left(q_{n}\times d_{n}\right)$-matrix. 

Now we consider $d_{n}-q_{n}$ vectors $e_{i}$, $i\in\underline{d_{n}-q_{n}},$
of length $d_{n}$ which are linearly independent of the vectors $\left(a_{m}^{n,k}\right)_{m\in D_{n}}$,
$k\in Q_{n}$. Via the Gram-Schmidt process we obtain $d_{n}-q_{n}$
orthonormal vectors $\left(c_{m}^{n,i}\right)_{m\in D_{n}}$, $i\in\underline{d_{n}-q_{n}}$,
of length $d_{n}$ which are orthonormal to $\left(a_{m}^{n,k}\right)_{m\in D_{n}}$,
$k\in Q_{n}$. We extend the vectors $\left(c_{m}^{n,i}\right)_{m\in D_{n}}$
to some of length $\underline{N^{n+2}}$ by $c_{m}^{n,i}=0$ if $m\in\underline{N^{n+2}}\backslash D_{n}$
and we define a matrix $\mathcal{C}{}_{n}:=\left(c_{m}^{n,k}\right)_{k\in\underline{d_{n}-q_{n}},m\in\underline{N^{n+2}}}$
of size $\left(d_{n}-q_{n}\right)\times N^{n+2}$ and $\mathcal{A}{}_{n}:=\left(a_{m}^{n,k}\right)_{k\in Q_{n},m\in\underline{N^{n+2}}}$
of size $q_{n}\times N^{n+2}$. So we obtain a matrix of size $d_{n}\times N^{n+2}$
by 

\begin{equation}
\mathcal{M}{}_{n}:=\left(\begin{array}{c}
\mathcal{A}{}_{n}\\
\mathcal{C}{}_{n}\end{array}\right).\label{eq:matrixone-1}\end{equation}

Now we turn to the construction of the coefficients for $\psi_{-n,k}$,
$n\in\N$, in (\ref{eq:def psi 2}). For each $-n$, $n\in\N$, we
define an orthonormal basis of $\C^{d_{-n}}$ in the following way. Consider
$\left(b_{m}^{n,k}\right)_{k\in Q_{-n},m\in D_{-n}}$. This is a $\left(q_{-n}\times d_{-n}\right)$-matrix.
Now we consider $d_{-n}-q_{-n}$ vectors which are linearly independent
of $\left(b_{m}^{n,k}\right)_{m\in D_{-n}}$. Via the Gram-Schmidt
process we obtain $d_{-n}-q_{-n}$ orthonormal vectors $\left(c_{m}^{-n,j}\right)_{m\in D_{-n}}$,
$j\in\underline{d_{-n}-q_{-n}}$, of length $d_{-n}$ which are orthonormal
to $\left(b_{m}^{n,i}\right)_{m\in D_{-n}}$. In the last step we
extend the vectors $\left(c_{m}^{-n,i}\right)_{m\in D_{-n}}$ to some
of length $N^{2}$ by defining $c_{m}^{-n,i}=0$ if $m\in\underline{N^{2}}\backslash D_{-n}$.
Now we define $\mathcal{D}{}_{n}:=\left(c_{m}^{-n,i}\right)_{i\in\underline{d_{-n}-q_{-n}},m\in\underline{N^{2}}}$
and $\mathcal{B}{}_{n}:=\left(b_{m}^{n,k}\right)_{k\in Q_{-n},m\in\underline{N^{2}}}$
such that \begin{equation}
\widetilde{\mathcal{M}}{}_{n}:=\left(\begin{array}{c}
\mathcal{B}{}_{n}\\
\mathcal{D}{}_{n}\end{array}\right)\label{eq:matrixtwo-1}\end{equation}
is a matrix of size $d_{-n}\times N^{2}$.

In the next step we show that we obtain indeed an orthonormal basis
with these mother wavelets given in (\ref{eq:def psi1}) and (\ref{eq:def psi 2}).
First we prove this for $n\in\N_{0}$. Recall that $W_{n}=V_{n+1}\ominus V_{n}$
for $n\in\N_{0}$. Consequently, $\cl\left(\bigcup_{n\in\N_{0}}W_{n}\cup V_{0}\right)=L^{2}(\mu)$
since for every $n\in\N_{0}$ it follows iteratively that $V_{n+1}=\bigoplus_{k=0}^{n}W_{k}\oplus V_{0}$.
Now we show that for fixed $n\in\N$ we have that $\left\{ \T^{l}\psi_{n,k}:k\in\underline{d_{n}-q_{n}},l\in\Z\right\} $
is an orthonormal basis of $W_{n}$. First we show the orthonormality. 

To show the orthonormality of $\T^{r}\psi_{n,k}$ and $\T^{s}\psi_{n,l}$,
$r,s\in\Z$, $k,l\in\underline{d_{n}-q_{n}}$, it is sufficient to
consider $\T^{r}\psi_{n,k}$ and $\psi_{n,l}$ since the operator
$\T$ is unitary. The orthonormality follows then from \begin{align*}
\langle\T^{r}\psi_{n,k}|\psi_{n,l}\rangle & =\Big\langle\sum_{m\in\underline{N^{n+2}}}c_{m}^{n,k}\U^{(n+1)}\T^{\lfloor\frac{m}{N}\rfloor+N^{n+1}r}\varphi_{(m)_{N}}\Big|\sum_{m\in\underline{N^{n+2}}}c_{m}^{n,l}\U^{(n+1)}\T^{\lfloor\frac{m}{N}\rfloor}\varphi_{(m)_{N}}\Big\rangle\\
 & =\Big\langle\sum_{m\in\underline{N^{n+2}}}c_{m}^{n,k}\T^{\lfloor\frac{m}{N}\rfloor+N^{n+1}r}\varphi_{(m)_{N}}\Big|\sum_{m\in\underline{N^{n+2}}}c_{m}^{n,l}\T^{\lfloor\frac{m}{N}\rfloor}\varphi_{(m)_{N}}\Big\rangle\\
 & =\sum_{m\in\underline{N^{n+2}}}\sum_{s\in\underline{N^{n+2}}}c_{m}^{n,k}\overline{c}_{s}^{n,l}\Big\langle\T^{N^{n+1}r+\lfloor\frac{m}{N}\rfloor}\varphi_{(m)_{N}}\Big|\T^{\lfloor\frac{s}{N}\rfloor}\varphi_{(s)_{N}}\Big\rangle\\
 & =\sum_{m\in\underline{N^{n+2}}}\sum_{s\in\underline{N^{n+2}}}c_{m}^{n,k}\overline{c}_{s}^{n,l}\cdot\delta_{(N^{n+1}r+\lfloor\frac{m}{N}\rfloor,\lfloor\frac{s}{N}\rfloor),((m)_{N},(s)_{N})}\\
 & =\delta_{r,0}\cdot\sum_{m\in\underline{N^{n+2}}}c_{m}^{n,k}\overline{c}_{m}^{n,l}\\
 & =\delta_{r,0}\cdot\delta_{k,l}.\end{align*}

In the next step we show that $V_{n+1}=V_{n}\oplus W_{n}$. We consider
a basis element of $V_{n+1}$ of the form $\U^{(n+1)}\T^{\lfloor\frac{k}{N}\rfloor}\varphi_{(k)_{N}}$,
$k\in\underline{N^{n+2}}$, and show that it is a linear combination
of functions $\U^{(n)}\T^{\lfloor\frac{l}{N}\rfloor}\varphi_{(l)_{N}}$
and $\psi_{n,m}$, $l\in\underline{N^{n+1}}$, $m\in\underline{d_{n}-q_{n}}$.
It is sufficient to consider only $k\in\underline{N^{n+2}}$ by Definition
\ref{def:MRA allgemein} (\ref{enu:-DefMRA5-1}). If $\U^{(n+1)}\T^{\lfloor\frac{k}{N}\rfloor}\varphi_{(k)_{N}}=0$
it is obvious satisfied. If $\U^{(n+1)}\T^{\lfloor\frac{k}{N}\rfloor}\varphi_{(k)_{N}}\neq0$,
$k\in\underline{N^{n+2}}$, it can be written as the following linear
combination:\begin{align*}
 & \U^{(n+1)}\T^{\lfloor\frac{k}{N}\rfloor}\varphi_{(k)_{N}}\\
= & \U^{(n+1)}\left(\sum_{m\in\underline{N^{n+2}}}\underbrace{\left(\sum_{l\in Q_{n}}\overline{a}_{k}^{n,l}a_{m}^{n,l}+\sum_{l\in\underline{d_{n}-q_{n}}}\overline{c}_{k}^{n,l}c_{m}^{n,l}\right)}_{=\delta_{k,m}}\T^{\lfloor\frac{m}{N}\rfloor}\varphi_{(m)_{N}}\right)\\
= & \U^{(n+1)}\left(\sum_{l\in Q_{n}}\overline{a}_{k}^{n,l}\sum_{m\in\underline{N^{n+2}}}a_{m}^{n,l}\T^{\lfloor\frac{m}{N}\rfloor}\varphi_{(m)_{N}}+\sum_{l\in\underline{d_{n}-q_{n}}}\overline{c}_{k}^{n,l}\sum_{m\in\underline{N^{n+2}}}c_{m}^{n,l}\T^{\lfloor\frac{m}{N}\rfloor}\varphi_{(m)_{N}}\right)\\
= & \sum_{l\in Q_{n}}\overline{a}_{k}^{n,l}\U^{(n)}\T^{\lfloor\frac{l}{N}\rfloor}\varphi_{(l)_{N}}+\sum_{l\in\underline{d_{n}-q_{n}}}\overline{c}_{k}^{n,l}\U^{(n+1)}\sum_{m\in\underline{N^{n+2}}}c_{m}^{n,l}\T^{\lfloor\frac{m}{N}\rfloor}\varphi_{(m)_{N}}\\
= & \sum_{l\in Q_{n}}\overline{a}_{k}^{n,l}\U^{(n)}T^{\lfloor\frac{l}{N}\rfloor}\varphi_{(l)_{N}}+\sum_{l\in\underline{d_{n}-q_{n}}}\overline{c}_{k}^{n,l}\psi_{n,l}.\end{align*}

If we consider $\T^{l}\psi_{n,k}$ and $\T^{r}\psi_{m,s}$ for $n,m\in\N$,
$n\neq m$, $l,r\in\Z$, $k\in\underline{d_{n}-q_{n}}$, $s\in\underline{d_{m}-q_{m}}$,
the orthonormality follows from $\T^{l}\psi_{n,k}\in W_{n}$, $\T^{r}\psi_{m,s}\in W_{m}$
and by the definition of $W_{n}$, $W_{m}$.

Now we consider the closed subspaces $V_{n}$ of $L^{2}(\mu)$ with
$n<0$ and show the analogous results. We show that for fixed $n\in\N$
$\left\{ \T^{N^{n}k}\psi_{-n,l}:\, l\in\underline{d_{-n}-q_{-n}},\, k\in\Z\right\} $
is an orthonormal basis of $W_{-n}=V_{-n+1}\ominus V_{-n}$. First
we show that any function $\U^{(-n+1)}\varphi_{j}$ can be written
as a linear combination of functions $\U^{(-n)}\varphi_{i}$ and $\psi_{-n,l}$,
$i\in\underline{N}$, $l\in\underline{d_{-n}-q_{-n}}$. This linear
combination is precisely\begin{align*}
 & \U^{(-n+1)}\varphi_{j}\\
= & \U^{(-n+1)}\left(\sum_{m\in\underline{N^{2}}}\left(\underbrace{\sum_{i\in Q_{-n}}\overline{b}_{j}^{n,i}b_{m}^{n,i}+\sum_{l\in\underline{d_{-n}-q_{-n}}}\overline{c}_{j}^{-n,l}c_{m}^{-n,l}}_{=\delta_{j,m}}\right)\T^{\lfloor\frac{m}{N}\rfloor}\varphi_{(m)_{N}}\right)\\
= & \U^{(-n+1)}\left(\sum_{i\in Q_{-n}}\overline{b}_{j}^{n,i}\sum_{m\in\underline{N^{2}}}b_{m}^{n,i}\T^{\lfloor\frac{m}{N}\rfloor}\varphi_{(m)_{N}}+\sum_{m\in\underline{N^{2}}}\sum_{l\in\underline{d_{-n}-q_{-n}}}\overline{c}_{j}^{-n,l}c_{m}^{-n,l}\T^{\lfloor\frac{m}{N}\rfloor}\varphi_{(m)_{N}}\right)\\
= & \sum_{i\in Q_{-n}}\overline{b}_{j}^{n,i}\U^{(-n)}\varphi_{i}+\sum_{l\in\underline{d_{-n}-q_{-n}}}\overline{c}_{j}^{-n,l}\U^{(-n+1)}\sum_{m\in\underline{N^{2}}}c_{m}^{-n,l}\T^{\lfloor\frac{m}{N}\rfloor}\varphi_{(m)_{N}}\\
= & \sum_{i\in Q_{-n}}\overline{b}_{j}^{n,i}\U^{(-n)}\varphi_{i}+\sum_{l\in\underline{d_{-n}-q_{-n}}}\overline{c}_{j}^{-n,l}\psi_{-n,l}.\end{align*}
We have to show the orthonormality only for $\psi_{-n,l}$ and $\psi_{-n,k}$
since $\T$ is a unitary operator and $\U^{(-n+1)}\T^{k}\varphi_{j}$
are mapped to orthonormal functions. For $\psi_{-n,l}$ and $\psi_{-n,k}$,
$l,k\in\underline{d_{-n}-q_{-n}}$, the orthonormality follows from\begin{align*}
\langle\psi_{-n,l}|\psi_{-n,k}\rangle & =\Big\langle\sum_{m\in\underline{N^{2}}}c_{m}^{-n,l}\U^{(-n+1)}\T^{\lfloor\frac{m}{N}\rfloor}\varphi_{(m)_{N}}\Big|\sum_{m\in\underline{N}^{2}}c_{m}^{-n,k}\U^{(-n+1)}\T^{\lfloor\frac{m}{N}\rfloor}\varphi_{(m)_{N}}\Big\rangle\\
 & =\sum_{m\in\underline{N^{2}}}c_{m}^{-n,l}\overline{c}_{m}^{-n,k}\\
 & =\delta_{l,k}.\end{align*}
Furthermore, it follows that $L^{2}(\mu)=\bigoplus_{k\in\Z}W_{k}$,
since we have shown before that $\cl\left(\bigcup_{n\in\N_{0}}W_{n}\cup V_{0}\right)=L^{2}(\mu)$.
Consequently, we have that \[
\left\{ \T^{l}\psi_{n,k}:n\in\Z,k\in\underline{d_{n}-q_{n}},l\in\Z\right\} \]
 is an ONB of $L^{2}(\mu)$.
\begin{rem}
For the proof of Corollary \ref{cor:onesided} we have to consider
the first part of the proof of Theorem \ref{thm:MRA} and show the
orthonormality between $\psi_{n,k}$ and $\varphi_{i}$ in addition,
which follows from the construction of the mother wavelets. 
\end{rem}

\subsection{\label{sub:Multiresolution-analysis-for}Abstract multiplicative
Multiresolution analysis }

In this section we want to consider how the general results simplify
if we impose the extra condition of a multiplicative MRA. 

Recall from the introduction that in the case of Definition \ref{def:MRA allgemein},
we say that we have a multiplicative MRA if there exists an operator
$\U$ such that $\U^{(n)}=\U^{n}$ for all $n\in\N$ and $\U^{(-n)}=\left(\U^{*}\right)^{n}$,
$n\in\N$. We then say $\left(\mu,\left(\left(\U\right)^{n},\left(\U^{*}\right)^{n}\right)_{n\in\N_{0}},\T\right)$
allows a two-sided multiplicative MRA. 

The key observation is contained in Lemma \ref{lem:Wn} which we prove
first.
\begin{proof}
[Proof of Lemma \ref{lem:Wn}] Recall that $\left\{ \T^{l}\psi_{0,k}:k\in\underline{d_{0}-N},\ l\in\Z\right\} $
is an orthonormal basis of $W_{0}$. We have $\psi_{0,k}=\sum_{m\in\underline{N^{2}}}c_{m}^{0,k}\U\T^{\lfloor\frac{m}{N}\rfloor}\varphi_{(m)_{N}}$
and we show that for fixed $n\in\N$, $\U^{n}W_{0}=W_{n}$. First
it follows that $\U^{n}\T^{m}\psi_{0,k}\in W_{n}\subset V_{n+1}$,
$n\in\N$, $m\in\Z$, since \[
\U^{n}\T^{m}\psi_{0,k}=\sum_{l\in\underline{N^{2}}}c_{l}^{0,k}\U^{n+1}\T{}^{\lfloor\frac{l}{N}\rfloor+Nm}\varphi_{(l)_{N}}\]
 and \begin{align*}
 & \langle\U^{n}\T^{m}\psi_{0,k}|\U^{n}\T^{r}\varphi_{i}\rangle\\
= & \langle\U^{n}\T^{m}\sum_{l\in\underline{N^{2}}}c_{l}^{0,k}\U\T^{\lfloor\frac{l}{N}\rfloor}\varphi_{(l)_{N}}|\U^{n}T^{r}\sum_{j\in\underline{N^{2}}}a_{j}^{0,i}\U\T^{\lfloor\frac{j}{N}\rfloor}\varphi_{(j)_{N}}\rangle\\
= & \langle\U^{n+1}\T^{Nm}\sum_{l\in\underline{N^{2}}}c_{l}^{0,k}\T^{\lfloor\frac{l}{N}\rfloor}\varphi_{(l)_{N}}|\U^{n+1}\T^{Nr}\sum_{j\in\underline{N^{2}}}a_{j}^{0,i}\T^{\lfloor\frac{j}{N}\rfloor}\varphi_{(j)_{N}}\rangle\\
= & \delta_{m,r}\cdot\sum_{l\in\underline{N^{2}}}c_{l}^{0,k}a_{l}^{0,i}=0.\end{align*}
Consequently, $\U^{n}W_{0}\subset W_{n}$. Now consider $\U^{n+1}\T^{m}\varphi_{j}\in V_{n+1}$,
$m\in\Z$, $j\in\underline{N}$, and we show that this can be written
as a linear combination of functions $\U^{n}\T^{l}\varphi_{i}$ and
$\U^{n}\T^{r}\psi_{0,k}$, $l,r\in\Z$, $i\in\underline{N}$, $k\in\underline{d_{0}-N}$,
by considering the scalar product. First we recall that $\U\T^{\lfloor\frac{k}{N}\rfloor}\varphi_{(k)_{N}}=\sum_{i\in\underline{N}}\overline{a}_{k}^{0,i}\varphi_{i}+\sum_{l\in\underline{d_{0}-N}}\overline{c}_{k}^{0,l}\psi_{0,l}$
for $k\in\underline{N^{2}}$ from the proof of Theorem \ref{thm:MRA},
and hence for $k\in\underline{N^{2}}$ \begin{align*}
1 & =\langle\U\T^{\lfloor\frac{k}{N}\rfloor}\varphi_{(k)_{N}}|\sum_{i\in\underline{N}}\overline{a}_{k}^{0,i}\varphi_{i}+\sum_{l\in\underline{d_{0}-N}}\overline{c}_{k}^{0,l}\psi_{0,l}\rangle\\
 & =\sum_{i\in\underline{N}}\overline{a}_{k}^{0,i}a_{k}^{0,i}+\sum_{l\in\underline{d_{0}-N}}\overline{c}_{k}^{0,l}c_{k}^{0,l}.\end{align*}
It follows that for $m\in\underline{N^{n+2}}$ with $m=k+N^{2}k_{1}$,
$k\in\underline{N^{2}}$, $k_{1}\in\underline{N^{n}}$, we have $\lfloor\frac{m}{N}\rfloor=\lfloor\frac{k}{N}\rfloor+Nk_{1}$
and $(m)_{N}=(k)_{N}$, and so \begin{align*}
 & \langle\U^{n+1}\T^{\lfloor\frac{k}{N}\rfloor+Nk_{1}}\varphi_{(k)_{N}}|\U^{n}\T^{k_{1}}\sum_{i\in\underline{N}}\overline{a}_{k}^{0,i}\varphi_{i}+\U^{n}\T^{k_{1}}\sum_{l\in\underline{d_{0}-N}}\overline{c}_{k}^{0,l}\psi_{0,l}\rangle\\
= & \langle\U^{n+1}\T^{\lfloor\frac{k}{N}\rfloor+Nk_{1}}\varphi_{(k)_{N}}|\U^{n}\T^{k_{1}}\sum_{i\in\underline{N}}\overline{a}_{k}^{0,i}\sum_{l\in\underline{N^{2}}}a_{l}^{0,i}\U\T^{\lfloor\frac{l}{N}\rfloor}\varphi_{(l)_{N}}\rangle\\
 & +\langle\U^{n+1}\T^{\lfloor\frac{k}{N}\rfloor+Nk_{1}}\varphi_{(k)_{N}}|\U^{n}\T^{k_{1}}\sum_{l\in\underline{d_{0}-N}}\overline{c}_{k}^{0,l}\sum_{i\in\underline{N^{2}}}c_{i}^{0,l}\U\T^{\lfloor\frac{i}{N}\rfloor}\varphi_{(i)_{N}}\rangle\\
= & \langle\T^{\lfloor\frac{k}{N}\rfloor+Nk_{1}}\varphi_{(k)_{N}}|\T^{Nk_{1}}\sum_{i\in\underline{N}}\overline{a}_{k}^{0,i}\sum_{l\in\underline{N^{2}}}a_{l}^{0,i}\T^{\lfloor\frac{l}{N}\rfloor}\varphi_{(l)_{N}}\rangle\\
 & +\langle\T^{\lfloor\frac{k}{N}\rfloor+Nk_{1}}\varphi_{(k)_{N}}|\T^{Nk_{1}}\sum_{l\in\underline{d_{0}-N}}\overline{c}_{k}^{0,l}\sum_{i\in\underline{N^{2}}}c_{i}^{0,l}\T^{\lfloor\frac{i}{N}\rfloor}\varphi_{(i)_{N}}\rangle\\
= & \sum_{i\in\underline{N}}\overline{a}_{k}^{0,i}a_{k}^{0,i}+\sum_{l\in\underline{d_{0}-N}}\overline{c}_{k}^{0,l}c_{k}^{0,l}\\
= & 1.\end{align*}
Now we notice that we can write any element $k\in\Z$ as $k=k_{0}+N^{n+2}l$
for some $k_{0}\in\underline{N^{n+2}}$ and $l\in\Z$. Consequently,
with $\U\T^{N}|_{V_{0}}=\T\U|_{V_{0}}$ we obtain the general result
for $\U^{n+1}\T^{k}\varphi_{j}$, $k\in\Z$, $j\in\underline{N}$. 

To obtain $W_{-n}=\left(\U^{*}\right)^{n}W_{-1}$, $W_{-1}=V_{0}\ominus V_{-1}$,
$n\in\N$, we can proceed as above. First we have from the proof of
Theorem \ref{thm:MRA} that $\psi_{-1,k}=\sum_{l\in\underline{N^{2}}}c_{l}^{-1,k}\T^{\lfloor\frac{l}{N}\rfloor}\varphi_{(l)_{N}}$
and that $\varphi_{j}$, $j\in\underline{N}$, can be represented
as $\varphi_{j}=\sum_{i\in\underline{N}}\overline{b}_{0,j}^{-1,i}\U^{*}\varphi_{i}+\sum_{l\in\underline{d_{-1}-N}}\overline{c}_{j}^{-1,l}\psi_{-1,l}$.
With these observations we obtain as above that \[
\langle\left(\U^{*}\right)^{n}\T^{m}\varphi_{j}|\left(\U^{*}\right)^{n-1}\T^{r}\psi_{-1,k}\rangle=0\]
 and \[
\langle\left(\U^{*}\right)^{n-1}\T^{k}\varphi_{j}|\left(\U^{*}\right)^{n}\sum_{i\in\underline{N}}\overline{b}_{j}^{-1,i}\varphi_{i}+\left(\U^{*}\right)^{n-1}\sum_{l\in\underline{d_{-1}-N}}\overline{c}_{j}^{-1,l}\psi_{-1,l}\rangle=1.\]
\end{proof}
\begin{rem}
If we have $\U\U^{*}=\id$, then $W_{0}=\U\left(W_{-1}\right)$. Notice
that $\U$ is not necessarily injective on $W_{-1}$. 
\end{rem}
Now we turn to the mother wavelets. 
\begin{rem}
\label{rem:motherwaveletsmulitplicative}If $\U^{(n)}=\U^{n}$, $\U^{(-n)}=\left(\U^{*}\right)^{n}$,
then we only consider the mother wavelets for $k\in\underline{d_{0}-N}$.
So \[
\psi_{0,k}=\U\sum_{l\in\underline{N^{2}}}c_{l}^{0,k}\T^{\lfloor\frac{l}{N}\rfloor}\varphi_{(l)_{N}},\]
where the coefficients are from (\ref{eq:onesided-1}) and we define
$\psi_{k}:=\psi_{0,k}$. 

For the negative indexed part of the construction we write for $k\in\underline{d_{-1}-N}$\[
\psi_{-,k}=\sum_{l\in\underline{N^{2}}}c_{l}^{-1,k}\T^{\lfloor\frac{l}{N}\rfloor}\varphi_{(l)_{N}}.\]

\end{rem}

\subsection{Translation completeness\label{sub:Translation-completeness}}

In the following we assume a stronger condition than (\ref{enu:-DefMRA5-1})
of Definition \ref{def:MRA allgemein}, namely that the father wavelets
are translation complete, i.e. for $j\in\underline{N}$\begin{equation}
\varphi_{j}\in\spn\U\left\{ \T^{j}\varphi_{i}:i\in\underline{N}\right\} .\label{eq:einfacheformphi}\end{equation}
This condition implies that there exist complex numbers $a_{i}^{0,j}$,
$i\in\underline{N}$, such that \[
\varphi_{j}=\sum_{i\in\underline{N}}a_{i}^{0,j}\U\T^{j}\varphi_{i}.\]
We would like to point out that this condition is also satisfied for
the particular case of MIM where the father wavelets $\varphi_{j}$,
$j\in\underline{N},$ are chosen to be the scaled characteristic functions
on the cylinder sets $[j]$ (see Section \ref{sec:Construction-of-a}).

We have that (\ref{eq:onesided-1}) takes the following form for $k\in\Z$,
$j\in\underline{N}$, $n\in\N$, \[
\U^{n}\T^{k}\varphi_{j}=\sum_{l\in\underline{N^{2}}}a_{l}^{0,j}\U^{n+1}\T^{Nk+\lfloor\frac{l}{N}\rfloor}\varphi_{(l)_{N}}\]
 and under condition (\ref{eq:einfacheformphi}) this simplifies to

\[
\U^{n}\T^{k}\varphi_{j}=\sum_{i\in\underline{N}}a_{i}^{0,j}\U^{n+1}\T^{Nk+j}\varphi_{i}.\]
To simplify the notation we set $a_{l}^{j}:=a_{l}^{0,j}$. We now
show that condition (\ref{eq:einfacheformphi}) allows us to simplify
the construction of the mother wavelets. 
\begin{lem}
\label{def:mother wavelets markov-1}Under condition (\ref{eq:einfacheformphi})
one possible choice of the matrix $\mathcal{M}{}_{0}$ in (\ref{eq:matrixone-1})
has a block structure consisting of $N$ blocks.\end{lem}
\begin{proof}
Define $Q^{k}:=\{j\in\underline{N}:\,\U\T^{k}\varphi_{j}\neq0\}$
and $q^{k}:=\card Q^{k}$ for each $k\in\underline{N}$. Then $\left(a_{j}^{k}\right)_{j\in Q^{k}}$
is a vector of length $q^{k}$ and we choose $q^{k}-1$ linearly independent
vectors to $\left(a_{j}^{k}\right)_{j\in Q^{k}}$ of length $q^{k}$.
Via the Gram-Schmidt process we obtain vectors $\left(c_{j}^{k,l}\right)_{j\in Q^{k}}$,
$l\in\underline{q^{k}}\backslash\{0\}$, orthonormal to $\left(a_{j}^{k}\right)_{j\in Q^{k}}$.
By setting $c_{j}^{k,l}=0$ if $j\in\underline{N}\backslash Q^{k}$
we extend the vectors $\left(c_{j}^{k,l}\right)_{j\in Q^{k}}$ to
$\left(c_{j}^{k,l}\right)_{j\in\underline{N}}$. Then $M_{k}:=\left(\begin{array}{c}
\left(a_{j}^{k}\right)_{j\in\underline{N}}\\
\left(c_{j}^{k,l}\right)_{l\in\underline{q^{k}}\backslash\{0\},j\in\underline{N}}\end{array}\right)$ is a matrix of size $q^{k}\times N$.

The matrix $\widehat{\mathcal{M}}_{0}=\left(h_{ij}\right)_{i\in\underline{q_{1}},j\in\underline{N^{2}}}$
given with the blocks $M_{k}$, $k\in\underline{N}$, by for $k=0$
\[
\left(h_{ij}\right)_{i\in\underline{q^{0}},j\in\underline{N}}=M_{0},\]
for $k\in\underline{N}\backslash\{0\}$ \[
\left(h_{ij}\right)_{i\in\underline{\sum_{l=0}^{k}q^{l}}\backslash\underline{\sum_{l=0}^{k-1}q^{l}},j\in\underline{(k+1)N}\backslash\underline{kN}}=M_{k}\]
and otherwise zero satisfies the conditions imposed on $\mathcal{M}_{0}$
in (\ref{eq:matrixone-1}), i.e. if we restrict the columns to those
in $D_{1}$, it is unitary, and $\widehat{\mathcal{M}_{0}}$ is of
size $q_{1}\times N^{2}$ since $\sum_{k\in\underline{N}}q^{k}=q_{1}$.
We notice that $\widehat{\mathcal{M}}_{0}$ is ordered in a different
way than $\mathcal{M}_{0}$, since the rows $\left(a_{j}^{k}\right)_{j\in Q^{k}}$
are not grouped in $\mathcal{M}_{0}$. \end{proof}
\begin{rem}
\label{rem:darstellung phi}$\ $
\begin{enumerate}
\item If $\U^{(n)}=\U^{n}$, $\U^{(-n)}=\left(\U^{*}\right)^{n}$ and (\ref{eq:einfacheformphi}),
the mother wavelets take the simpler form for $k=0$, $l\in\underline{q^{0}}\backslash\{0\}$
and for $k\in\underline{N}\backslash\{0\}$, $l\in\underline{\sum_{i=0}^{k}q^{i}}\backslash\underline{\sum_{i=0}^{k-1}q^{i}}$,
as \[
\psi_{l}=\sum_{j\in\underline{N}}c_{j}^{k,l}\U\T^{k}\varphi_{j},\]
where the coefficients are as constructed in Lemma \ref{def:mother wavelets markov-1}.
For negative indexed part we define for $k\in\underline{d_{-1}-N}$\[
\psi_{-,k}=\sum_{l\in\underline{N^{2}}}c_{l}^{-1,k}\T^{\lfloor\frac{l}{N}\rfloor}\varphi_{(l)_{N}}.\]

\item In the case of (\ref{eq:einfacheformphi}), or the slightly weaker
statement \begin{equation}
\U^{(n)}\T^{\lfloor\frac{k}{N}\rfloor}\varphi_{(k)_{N}}=\sum_{i\in\underline{N}}a_{i}^{n,k}\U^{(n+1)}\T^{N\lfloor\frac{k}{N}\rfloor+(j)_{N}}\varphi_{i}\label{eq:simpleform Un}\end{equation}
 we can obtain the coefficients for the mother wavelets by constructing
for each $k\in\underline{N^{n}}$ with $\U^{(n)}\T^{k}\varphi_{j}\not\neq0$
for at least one $j\in\underline{N}$ a matrix of size $q^{n,k}\times q^{n,k}$,
where $q^{n,k}=\card\{j\in\underline{N}:\,\U^{(n)}\T^{k}\varphi_{j}\not\neq0\}$
instead of one unitary matrix of size $d_{n}\times d_{n}$. In this
way we need at most $N^{n}$ matrices on the scale $n\in\N$. 
\end{enumerate}
\end{rem}
Now we turn to a correspondence to the construction of a wavelet basis
for MIM. The next proposition shows how the incidence matrix of MIM
plays a role in the MRA. 
\begin{prop}
\label{pro:transition matrix}In the case of $\U^{(n)}=\left(\U^{(1)}\right)^{n}$,
$n\in\N_{0}$, (\ref{eq:einfacheformphi}) and if it further holds
that $a_{i}^{0,j}\neq0$ if and only if $\U\T^{j}\varphi_{i}\neq0$,
$i,j\in\underline{N}$, then we have for $n\in\N$, $k\in\Z$, $\U^{n}\T^{k}\varphi_{j}\neq0$
if and only if for all $i=0,\dots,n-2$, $\U\T^{k_{i+1}}\varphi_{k_{i}}\not\neq0$
and $\U\T^{k_{0}}\varphi_{j}\neq0$, where $k=\sum_{i=0}^{n-1}k_{i}N^{i}+lN^{n}$,
$k_{i}\in\underline{N}$ , $i\in\underline{n}$, and $l\in\Z$. \end{prop}
\begin{proof}
We prove this for $k=k_{0}+Nk_{1}$, $k_{0},k_{1}\in\underline{N}$.
The general result follows iteratively. Notice that $\U^{2}\T^{k_{0}+Nk_{1}}\varphi_{j}=\U\T^{k_{1}}\left(\U\T^{k_{0}}\varphi_{j}\right)$.
Consequently, from $\U^{2}\T^{k_{0}+Nk_{1}}\varphi_{j}\neq0$ it follows
that $\U\T^{k_{0}}\varphi_{j}\neq0$. Besides we have that \[
\U\T^{k_{1}}\varphi_{k_{0}}=\U\T^{k_{1}}\sum_{i\in\underline{N}}a_{i}^{k_{0}}\U\T^{k_{0}}\varphi_{i}=\U^{2}\T^{Nk_{1}+k_{0}}\sum_{i\in\underline{N}}a_{i}^{k_{0}}\varphi_{i}\neq0\]
 if $\U^{2}\T^{k_{0}+Nk_{1}}\varphi_{j}\neq0$.

If we assume that $\U\T^{k_{1}}\varphi_{k_{0}}\neq0$ and $\U\T^{k_{0}}\varphi_{j}\neq0$
then\begin{align*}
\U^{2}\T^{Nk_{1}+k_{0}}\varphi_{j} & =\U\T^{k_{1}}\U\T^{k_{0}}\varphi_{j}\\
 & =\left(a_{j}^{k_{0}}\right)^{-1}\U\T^{k_{1}}\left(\varphi_{k_{0}}-\sum_{i\in\underline{N}\backslash\{j\}}a_{i}^{k_{0}}\U\T^{k_{0}}\varphi_{i}\right)\\
 & =\left(a_{j}^{k_{0}}\right)^{-1}\left(\U\T^{k_{1}}\varphi_{k_{0}}-\sum_{i\in\underline{N}\backslash\{j\}}a_{i}^{k_{0}}\U^{2}\T^{Nk_{1}+k_{0}}\varphi_{i}\right)\\
 & \neq0,\end{align*}
since $\U\T^{k_{1}}\varphi_{k_{0}}-\sum_{i\in\underline{N}\backslash\{j\}}a_{i}^{k_{0}}\U^{2}\T^{Nk_{1}+k_{0}}\varphi_{i}=a_{j}^{k_{0}}\U T^{Nk_{1}+k_{0}}\varphi_{j}\neq0$.\end{proof}
\begin{rem}
\label{rem:transition matrix}We can show the same result if for some
$c\in\R$, we have $\U^{(n)}\T^{k}\varphi_{j}=c\left(\U^{(1)}\right)^{n}\T^{k}\varphi_{j}$
for all $n\in\N$, $k\in\underline{N^{n}}$ and $j\in\underline{N}$
and (\ref{eq:einfacheformphi}), where $c$ may depend on $n,k,j$.
\end{rem}
Under the conditions of Proposition \ref{pro:transition matrix} we
can give a $N\times N$ matrix $A$, which coincides with the incidence
matrix in the case of MIM given by $A=\left(A_{ij}\right)_{i,j\in\underline{N}}$
with \[
A_{ij}:=\begin{cases}
0, & \text{if }\,\U\T^{i}\varphi_{j}=0,\\
1, & \text{else.}\end{cases}\]

\section{\label{sec:Construction-of-a}Applications to Markov Interval Maps}

\subsection{Multiresolution Analysis for MIM}

Now we apply the results of Section \ref{sec:Multiresolution-analysis}
to Markov Interval Maps. More precisely, we construct a wavelet basis
on the $L^{2}$-space of a limit set of a Markov Interval Map translated
by $\Z$ with respect to a measure. First we consider the case where
we do not have any relation between the measures of $\nu_{\Z}\left(\left[ij\right]\right)$
and $\nu_{\Z}\left(\left[i\right]\right)$, $\nu_{\Z}\left(\left[j\right]\right)$.
In this case we cannot define only one operator $U$, but on each
scale $n\in\Z$ we consider a different operator $U^{(n)}$. Consequently,
we obtain a family of operators $\left(U^{(n)}\right)_{n\in\Z}$. 

The operators $\left(U^{(n)}\right)_{n\in\Z}$ and $T$ are defined
in (\ref{eq:Def U^n}), (\ref{eq:Def U^-n}) and (\ref{eq:Def T})
respectively. 
\begin{rem}
\label{rem:U(1)U(1) U(2)}$\ $
\begin{enumerate}
\item Notice that in general we have $U^{(1)}U^{(1)}\neq U^{(2)}$ since
the multiplicative constant $\sqrt{\frac{\nu_{\Z}([j])}{\nu_{\Z}([ij])}\frac{\nu_{\Z}([k)}{\nu_{\Z}([kl])}}$
for $U^{(1)}U^{(1)}$ and $\sqrt{\frac{\nu_{\Z}([j])}{\nu_{\Z}([kij])}}$
for $U^{(2)}$ on the cylinder sets may differ. 
\item The operator $T$ is unitary. 
\item The operators $\left(U^{(n)}\right)_{n\in\Z}$ are well defined, namely
for $f\in L^{2}(\nu_{\Z})$ we have $U^{(n)}f\in L^{2}(\nu_{\Z})$. 
\end{enumerate}
\end{rem}
Define the $N$ \textit{father wavelets} as $\varphi_{i}:=\left(\mu([i])\right)^{-1/2}\mathbbm{1}_{[i]}$
for $i\in\underline{N}$.
\begin{rem}
Notice that for $\omega\in\Sigma_{A}^{n}$, $j\in\underline{N}$ and
$k\in\Z$ with $k=\sum_{i=0}^{n-1}\omega_{n-1-i}N^{i}+N^{n}l$, $l\in\Z$,
we have\begin{equation}
\begin{array}{ccc}
U^{(n)}T^{k}\varphi_{j} & = & \begin{cases}
0, & \text{if }A_{\omega_{n-1}j}=0,\\
\left(\nu_{\Z}([\omega j])\right)^{-1/2}T^{l}\mathbbm{1}_{[\omega j]}, & \text{else}.\end{cases}\end{array}\label{eq:funkUnTkphij}\end{equation}

\end{rem}
Now we turn to the proof of the properties of $\left(U^{(n)}\right)_{n\in\Z}$
and $T$ stated in Proposition \ref{pro:The-operatorsUT-1}.
\begin{proof}
[Proof of Proposition  \ref{pro:The-operatorsUT-1}] 

ad (\ref{enu:(1)TU-1}): Let $n\in\N$, $f\in L^{2}(\nu_{\Z})$, $x\in\R$,
then \begin{align*}
 & TU^{(n)}f(x)\\
= & \sum_{k\in\Z}\sum_{\omega\in\Sigma_{A}^{n}}\sum_{j\in\underline{N}}\sqrt{\frac{\nu_{\Z}([j])}{\nu_{\Z}([\omega j])}}\mathbbm{1}_{[\omega j]}(x-1-k)\cdot f\left(\tau_{\omega}^{-1}(x-1-k)+\sum_{i=0}^{n-1}\omega_{n-1-i}N^{i}+N^{n}k\right)\\
= & \sum_{l\in\Z}\sum_{\omega\in\Sigma_{A}^{n}}\sum_{j\in\underline{N}}\sqrt{\frac{\nu_{\Z}([j])}{\nu_{\Z}([\omega j])}}\mathbbm{1}_{[\omega j]}(x-l)\cdot f\left(\tau_{\omega}^{-1}(x-l)+\sum_{i=0}^{n-1}\omega_{n-1-i}N^{i}+N^{n}l-N^{n}\right)\\
= & U^{(n)}T^{N^{n}}f(x).\end{align*}

ad (\ref{enu:(3)TU-1-1}): Let $i\in\underline{N}$, $x\in\R$, then
\begin{align*}
 & \varphi_{i}(x)\\
= & \left(\nu_{\Z}([i])\right)^{-1/2}\sum_{j\in\underline{N}}\mathbbm{1}_{[ij]}(x)\\
= & \sum_{j\in\underline{N}}\sqrt{\frac{\nu_{\Z}([ij])}{\nu_{\Z}([i])}\frac{\nu_{\Z}([j])}{\nu_{\Z}([ij])}}\cdot\left(\mu([j])\right)^{-1/2}\mathbbm{1}_{[ij]}(x)\\
= & \sum_{j\in\underline{N}}\sqrt{\frac{\nu_{\Z}([ij])}{\nu_{\Z}([i])}}\sum_{k\in\underline{N}}\sqrt{\frac{\nu_{\Z}([k])}{\nu_{\Z}([ik])}}\mathbbm{1}_{[ik]}(x)\cdot\varphi_{j}\left(\tau_{i}^{-1}(x)\right)\\
= & U^{(1)}T^{i}\sum_{j\in\underline{N}}\sqrt{\frac{\nu_{\Z}([ij])}{\nu_{\Z}([i])}}\varphi_{j}(x).\end{align*}

ad (\ref{enu:(2)Znt-1}): Notice that for $n\in\N$, $l\in\underline{N}$,
$k\in\Z$, \[
U^{(-n)}\varphi_{l}(x)=\sum_{\omega\in\Sigma_{A}^{n}:\omega_{0}=l}\sum_{j\in\underline{N}}\sqrt{\frac{\nu_{\Z}([\omega j])}{\nu_{\Z}([l])}}\varphi_{j}\left(x-\sum_{i=0}^{n-1}\omega_{n-1-i}N^{i}\right)\]
and \[
U^{(-n)}T^{k}\varphi_{l}(x)=\sum_{\omega\in\Sigma_{A}^{n}:\omega_{0}=l}\sum_{j\in\underline{N}}\sqrt{\frac{\nu_{\Z}([\omega j])}{\nu_{\Z}([l])}}\varphi_{j}\left(x-\sum_{i=0}^{n-1}\omega_{n-1-i}N^{i}-N^{n}k\right).\]
Consequently, $T^{N^{n}k}U^{(-n)}\varphi_{j}=U^{(-n)}T^{k}\varphi_{j}$
for all $k\in\Z$, $n\in\N$, $j\in\underline{N}$. 

ad (\ref{enu:(4)TU-1}): Let $n\in\N$ and $k=\sum_{i=0}^{n-1}\omega_{n-1-i}N^{i}+N^{n}k_{1}$,
$\omega\in\Sigma_{A}^{n}$, $k_{1}\in\Z$ and $l=\sum_{i=0}^{n-1}\tilde{\omega}_{n-1-i}N^{i}+N^{n}l_{1}$,
$\tilde{\omega}\in\Sigma_{A}^{n}$, $l_{1}\in\Z$ and $A_{\omega_{n-1}i}=1$
and $A_{\tilde{\omega}_{n-1}j}=1$ for $i,j\in\underline{N}$ then\begin{align*}
\langle U^{(n)}T^{k}\varphi_{i}|U^{(n)}T^{l}\varphi_{j}\rangle & =\langle\left(\nu_{\Z}([\omega i])\right)^{-1/2}T^{k_{1}}\mathbbm{1}_{[\omega i]}|\left(\nu_{\Z}([\tilde{\omega}j])\right)^{-1/2}T^{l_{1}}\mathbbm{1}_{[\tilde{\omega}j]}\rangle\\
 & =\delta_{k_{1},l_{1}}\delta_{(\omega,i),(\tilde{\omega},j)}.\end{align*}
Otherwise, we have $U^{(n)}T^{k}\varphi_{i}=0$ or $U^{(n)}T^{l}\varphi_{j}=0$.

Furthermore for $n\in\N$, $k,j\in\Z$, $i,m\in\underline{N}$, we
have \begin{align*}
 & \langle U^{(-n)}T^{k}\varphi_{i}|U^{(-n)}T^{l}\varphi_{m}\rangle\\
= & \Big\langle\sum_{\omega\in\Sigma_{A}^{n}:\omega_{0}=i}\sum_{j\in\underline{N}}\sqrt{\frac{\nu_{\Z}([\omega j])}{\nu_{\Z}([i])}}T^{\sum_{i=0}^{n-1}\omega_{n-1-i}N^{i}+N^{n}k}\varphi_{j}\\
 & \ \ |\sum_{\omega\in\Sigma_{A}^{n}:\omega_{0}=m}\sum_{j\in\underline{N}}\sqrt{\frac{\nu_{\Z}([\omega j])}{\nu_{\Z}([m])}}T^{\sum_{i=0}^{n-1}\omega_{n-1-i}N^{i}+N^{n}l}\varphi_{j}\Big\rangle\\
= & \delta_{k,l}\cdot\delta_{i,m}\cdot\sum_{\omega\in\Sigma_{A}^{n}:\omega_{0}=i}\sum_{j\in\underline{N}}\frac{\nu_{\Z}([\omega j])}{\nu_{\Z}([i])}\\
= & \delta_{(k,i),(l,m)},\end{align*}
where we used in the second equality that $\langle T^{k}\varphi_{j}|T^{l}\varphi_{i}\rangle=\delta_{(k,j),(l,i)}$.

ad (\ref{enu:(5)UZ-1}): Let $n\in\N$, $f\in L^{2}(\nu_{\Z})$, $x\in\R$,
then \begin{align*}
 & U^{(n)}U^{(-n)}f(x)\\
= & \sum_{l\in\Z}\sum_{\widetilde{\omega}\in\Sigma_{A}^{n}}\sum_{r\in\underline{N}}\sqrt{\frac{\nu_{\Z}([r])}{\nu_{\Z}([\widetilde{\omega}r])}}\1_{[\widetilde{\omega}r]}(x-l)\sum_{k\in\Z}\sum_{\omega\in\Sigma_{A}^{n}}\sum_{j\in\underline{N}}\sqrt{\frac{\nu_{\Z}([\omega j])}{\nu_{\Z}([j])}}\\
 & \ \ \ \1_{[j]}\left(\tau_{\widetilde{\omega}}^{-1}(x-l)+\sum_{i=0}^{n-1}\widetilde{\omega}_{n-1-i}N^{i}+N^{n}l-\sum_{i=0}^{n-1}\omega_{n-1-i}N^{i}-N^{n}k\right)\\
 & \ \ \ f\left(\tau_{\omega}\left(\tau_{\widetilde{\omega}}^{-1}(x-l)+\sum_{i=0}^{n-1}\widetilde{\omega}_{n-1-i}N^{i}+N^{n}l-\sum_{i=0}^{n-1}\omega_{n-1-i}N^{i}-N^{n}k\right)+k\right)\\
= & \sum_{k\in\Z}\sum_{\omega\in\Sigma_{A}^{n}}\sum_{j\in\underline{N}}\1_{[\omega j]}(x-k)\cdot f(x)\\
= & f(x),\end{align*}
where we used in the third equality that $i=r,$ $\omega=\widetilde{\text{\ensuremath{\omega}}}$
and $k=l$ since otherwise it is zero. 

ad (\ref{enu:(6)ZU-1}): For $n\in\N$, $k\in\Z$, $j\in\underline{N}$,
$x\in\R$, with $U^{(n)}T^{k}\varphi_{j}\neq0$, there is $\omega\in\Sigma_{A}^{n}$,
$l\in\Z$, with $k=\sum_{i=0}^{n-1}\omega_{n-1-i}N^{i}+N^{n}l$ and
so \begin{align*}
U^{(-n)}U^{(n)}T^{k}\varphi_{j}(x) & =U^{(-n)}\left(\left(\nu_{\Z}([\omega j])\right)^{-1/2}T^{l}\mathbbm{1}_{[\omega j]}(x)\right)\\
 & =T^{N^{n}l}U^{(-n)}\left(\left(\nu_{\Z}([\omega j])\right)^{-1/2}\mathbbm{1}_{[\omega j]}(x)\right)\\
 & =T^{N^{n}l}T^{\sum_{i=0}^{n-1}\omega_{n-1-i}N^{i}}\left(\nu_{\Z}([j])\right)^{-1/2}\mathbbm{1}_{[j]}(x)\\
 & =T^{k}\varphi_{j}(x).\end{align*}
\end{proof}
\begin{rem}
We further notice that for $n\in\N$, $x\in\R$, we have $f\in L^{2}(\nu_{\Z})$,
\[
U^{(-n)}U^{(n)}f(x)=\sum_{k\in\Z}\sum_{\omega j\in\Sigma_{A}^{n+1}}\mathbbm{1}_{[j]}\left(x-\sum_{i=0}^{n-1}\omega_{n-1-i}N^{i}-N^{n}k\right)\cdot f(x),\]
and consequently, in general we do not have $U^{(-n)}U^{(n)}=\id$.
\end{rem}
Now we can turn to the proof of Theorem \ref{thm:MRa}. 
\begin{proof}
[Proof of Theorem   \ref{thm:MRa}] We show the properties (\ref{enu:DefMRA2-1-1})
to (\ref{enu:DefMRA5-2}) of Definition \ref{def:MRA allgemein} with
the father wavelets $\varphi_{i}=\left(\nu([i]\right)^{-1/2}\mathbbm{1}_{[i]}$,
$i\in\underline{N}$. 

We define the closed subspaces of $L^{2}(\nu_{\Z})$ for $j\in\N$
as \begin{align*}
V_{0} & :=\cl\spn\left\{ T^{k}\varphi_{i}:\, k\in\Z,\, i\in\underline{N}\right\} ,\\
V_{j} & :=\cl\spn\left\{ U^{(j)}T^{k}\varphi_{i}:\, k\in\Z,\, i\in\underline{N}\right\} .\end{align*}

ad (\ref{enu:DefMRA3-2}): By the definition of $V_{j}$ we obviously
have that $\left\{ U^{(j)}T^{k}\varphi_{i}:k\in\Z,i\in\underline{N}\right\} $
spans $V_{j}$, $j\in\Z$. The orthonormality follows from Proposition
\ref{pro:The-operatorsUT-1} (\ref{enu:(4)TU-1}). 

ad (\ref{enu:DefMRA4-2}): We notice that for $\omega\in\Sigma_{A}^{n}$,
$j\in\underline{N}$ and $k\in\Z$ with $k=\sum_{i=0}^{n-1}\omega_{n-1-i}N^{i}+N^{n}l$,
$l\in\Z$, we have \begin{align*}
U^{(n)}T^{k}\varphi_{j} & =\sum_{i\in\underline{N}}\sqrt{\frac{\nu_{\Z}([\omega ji])}{\nu_{\Z}([\omega j])}}U^{(n+1)}T^{Nk+j}\varphi_{i}.\end{align*}
If there is not such an $\omega\in\Sigma_{A}^{n}$ so that $k=\sum_{i=0}^{n-1}\omega_{n-1-i}N^{i}+N^{n}l$,
$l\in\Z$, then $U^{(n)}T^{k}\varphi_{j}=0$. 

ad (\ref{enu:DefMRA2-1-1}): Notice that for $n\in\N$, $k\in\Z$
and $i\in\underline{N}$ we obtain with Proposition \ref{pro:The-operatorsUT-1}
(\ref{enu:(2)Znt-1}) and (\ref{enu:(3)TU-1-1}) that\begin{align*}
U^{(n)}T^{k}\varphi_{i} & =U^{(n)}T^{k}U^{(1)}T^{i}\sum_{j\in\underline{N}}\sqrt{\frac{\nu_{\Z}([i])}{\nu_{\Z}([ij])}}\varphi_{j}\\
 & =U^{(n)}U^{(1)}T^{Nk+i}\sum_{j\in\underline{N}}\sqrt{\frac{\nu_{\Z}([i])}{\nu_{\Z}([ij])}}\varphi_{j}.\end{align*}
So it follows that $V_{n}\subset V_{n+1}$ by Remark \ref{rem:U(1)U(1) U(2)}
(1). 

ad (\ref{enu:DefMRA2-2-1}): First we notice that $X$ is either totally
disconnected or we can consider $X$ as an interval in $[0,1]$. Furthermore,
every characteristic function on a cylinder $[\omega]\subset\Sigma_{A}$
can be obtained by $U^{(n)}T^{k}\varphi_{j}$, $n\in\N_{0}$, $k\in\Z$,
$j\in\underline{N}$. Thus, we are left to show that $\left\{ T^{k}\mathbbm{1}_{[\omega]}:\, k\in\Z,\,\omega\in\Sigma_{A}^{*}\right\} $
is dense in $L^{2}(\nu_{\Z})$. 

If $X$ is totally disconnected, it follows by the Stone-Weierstrass
Theorem that \[
\left\{ T^{k}\mathbbm{1}_{[\omega]}:\, k\in\Z,\,\omega\in\Sigma_{A}^{*}\right\} \]
 is dense in $C(R,\mathbb{C})$, see e.g. \cite{KeStaStr07}. Besides
it is well known that $C(R,\mathbb{C})$ is dense in $L^{2}(\nu_{\Z})$
and so $\cl\spn\left\{ T^{k}\mathbbm{1}_{[\omega]}:\, k\in\Z,\,\omega\in\Sigma_{A}^{*}\right\} =L^{2}(\nu_{\Z})$. 

If $X=[a,b]$, notice that every interval $I\subset[0,1]$ can be
approximated by $\tau_{\omega}(X)$, $\omega\in\Sigma_{A}^{*}$. Hence
$\tau_{\omega}(X)$, $\omega\in\Sigma_{A}^{*}$, generates $\mathcal{B}$,
thus every element $A\in\mathcal{B}$ can be approximated by elements
of $\left\{ \tau_{\omega}(X):\,\omega\in\Sigma_{A}^{*}\right\} $.
Consequently, every elementary function can be approximated by functions
$\mathbbm{1}_{\tau_{\omega}(X)}$ and so all functions in $L^{2}(\nu_{\Z})$
can be approximated by elements of \[
\left\{ T^{k}\mathbbm{1}_{[\omega]}:\,\omega\in\Sigma_{A}^{*},k\in\Z\right\} =\left\{ U^{(n)}T^{l}\varphi_{i}:\, n\in\N_{0},\, l\in\Z,\, i\in\underline{N}\right\} .\]
Consequently, $\cl\bigcup_{k\in\N}V_{n}=L^{2}(\nu_{\Z})$ .

ad (\ref{enu:DefMRA5-2}): This follows from Proposition \ref{pro:The-operatorsUT-1}
(1) and (2).
\end{proof}
Next we prove the forward direction of Theorem \ref{thm:MRA-1}. The
backward direction will be shown in Section \ref{sub:MRA-for-Markov}.
\begin{proof}
[Proof of Theorem \ref{thm:MRA-1}  {}``$\Longrightarrow$'']

We assume that $\left(\nu_{\Z},\left(U^{(n)}\right)_{n\in\Z},T\right)$
allows a two-sided MRA with the father wavelets $\varphi_{i}=\left(\nu_{\Z}([i]\right)^{-1/2}\mathbbm{1}_{[i]}$.
Then in particular, it holds by (\ref{enu:DefMRA4-2}) of Definition
\ref{def:MRA allgemein} that for $n\in\N$ \[
U^{(-n)}\left\{ \varphi_{i}:i\in\underline{N}\right\} \subset\spn U^{(-n+1)}\left\{ T^{k}\varphi_{i}:i\in\underline{N},k\in\underline{N}\right\} .\]
We further notice that for $n\in\N$, $k,i\in\underline{N}$, \[
U^{(-n)}\varphi_{k}=\sum_{\omega\in\Sigma_{A}^{n}:\omega_{0}=k}\sum_{j\in\underline{N}}\sqrt{\frac{\nu_{\Z}([\omega j])}{\nu_{\Z}([k])}}T^{\sum_{l=0}^{n-1}\omega_{n-1-l}N^{l}}\varphi_{j}\]
 and \[
U^{(-n+1)}T^{k}\varphi_{i}=\sum_{\omega\in\Sigma_{A}^{n-1}:\omega_{0}=i}\sum_{j\in\underline{N}}\sqrt{\frac{\nu_{\Z}([\omega j])}{\nu_{\Z}([i])}}T^{\sum_{l=0}^{n-2}\omega_{n-2-l}N^{l}+N^{n-1}k}\varphi_{j}.\]
From the precise from of $U^{(-n)}\varphi_{k}$ and $U^{(-n+1)}T^{m}\varphi_{i}$,
$n\in\N$, $k,m,i\in\underline{N}$, it follows that $\langle U^{(-n)}\varphi_{k}|U^{(-n+1)}T^{m}\varphi_{i}\rangle\neq0$
only if $m=k$ since\begin{align*}
 & \langle U^{(-n)}\varphi_{k}|U^{(-n+1)}T^{m}\varphi_{i}\rangle\\
= & \langle\sum_{\omega\in\Sigma_{A}^{n}:\omega_{0}=k}\sum_{j\in\underline{N}}\sqrt{\frac{\nu_{\Z}([\omega j])}{\nu_{\Z}([k])}}T^{\sum_{l=0}^{n-1}\omega_{n-1-l}N^{l}}\varphi_{j}\\
 & \ \ \ |\sum_{\omega\in\Sigma_{A}^{n-1}:\omega_{0}=i}\sum_{j\in\underline{N}}\sqrt{\frac{\nu_{\Z}([\omega j])}{\nu_{\Z}([i])}}T^{\sum_{l=0}^{n-2}\omega_{n-2-l}N^{l}+N^{n-1}m}\varphi_{j}\rangle\\
= & \sum_{\omega\in\Sigma_{A}^{n}:\omega_{0}=k}\sum_{j_{1}\in\underline{N}}\sum_{\widetilde{\omega}\in\Sigma_{A}^{n-1}:\widetilde{\omega}_{0}=i}\sum_{j_{2}\in\underline{N}}\sqrt{\frac{\nu_{\Z}([\omega j_{1}])}{\nu_{\Z}([k])}}\sqrt{\frac{\nu_{\Z}([\widetilde{\omega}j_{2}])}{\nu_{\Z}([i])}}\\
 & \ \ \ \ \ \langle T^{\sum_{l=0}^{n-1}\omega_{n-1-l}N^{l}}\varphi_{j_{1}}|T^{\sum_{l=0}^{n-2}\widetilde{\omega}_{n-2-l}N^{l}+N^{n-1}m}\varphi_{j_{2}}\rangle\\
= & \delta_{m,k}\sum_{j\in\underline{N}}\sum_{\omega\in\Sigma_{A}^{n-1}:\omega_{0}=i}\sqrt{\frac{\nu_{\Z}([k\omega j])}{\nu_{\Z}([k])}}\sqrt{\frac{\nu_{\Z}([\omega j])}{\nu_{\Z}([i])}},\end{align*}
where we used in the third equality the property of Proposition \ref{pro:The-operatorsUT-1}
(\ref{enu:(4)TU-1}), namely $\langle T^{k}\varphi_{j_{1}}|T^{l}\varphi_{j_{2}}\rangle=\delta_{(k,j_{1}),(l,j_{2})}$
for any $k,l\in\Z$ and $j_{1},j_{2}\in\underline{N}$.

As a consequence of (\ref{enu:DefMRA3-2}), (\ref{enu:DefMRA4-2})
of Definition \ref{def:MRA allgemein} and the observation above it
follows that for every $n\in\N$, $k\in\underline{N}$ there exist
unique $\left(\alpha_{i}^{n,k}\right)_{i\in\underline{N}}\in\C^{N}$
such that\begin{align*}
U^{(-n)}\varphi_{k} & =\sum_{i\in\underline{N}}\alpha_{i}^{n,k}U^{(-n+1)}T^{k}\varphi_{i}\\
 & =\sum_{i\in\underline{N}}\alpha_{i}^{n,k}\sum_{\omega\in\Sigma_{A}^{n-1}:\omega_{0}=i}\sum_{j\in\underline{N}}\sqrt{\frac{\nu_{\Z}([\omega j])}{\nu_{\Z}([i])}}T^{\sum_{l=0}^{n-2}\omega_{n-2-l}N^{l}+N^{n-1}k}\varphi_{j}.\end{align*}
On the other hand, from the precise form of $U^{(-n)}\varphi_{k}$
it follows that 

\begin{align*}
U^{(-n)}\varphi_{k} & =\sum_{\omega\in\Sigma_{A}^{n}:\omega_{0}=k}\sum_{j\in\underline{N}}\sqrt{\frac{\nu_{\Z}([\omega j])}{\nu_{\Z}([k])}}T^{\sum_{l=0}^{n-1}\omega_{n-1-l}N^{l}}\varphi_{j}\\
 & =\sum_{\omega\in\Sigma_{A}^{n-1}}\sum_{j\in\underline{N}}\sqrt{\frac{\nu_{\Z}([k\omega j])}{\nu_{\Z}([k])}}T^{\sum_{l=0}^{n-2}\omega_{n-2-l}N^{l}+N^{n-1}k}\varphi_{j}\\
 & =\sum_{i\in\underline{N}}\sum_{\omega\in\Sigma_{A}^{n-1}:\omega_{0}=i}\sum_{j\in\underline{N}}\sqrt{\frac{\nu_{\Z}([k\omega j])}{\nu_{\Z}([k])}}T^{\sum_{l=0}^{n-2}\omega_{n-2-l}N^{l}+N^{n-1}k}\varphi_{j}.\end{align*}
By comparing the coefficients it follows that for every $\omega\in\Sigma_{A}^{n-1}$,
$\omega_{0}=i$, we have $\alpha_{i}^{n,k}\sqrt{\frac{\nu_{\Z}([\omega j])}{\nu_{\Z}([i])}}=\sqrt{\frac{\nu_{\Z}([k\omega j])}{\nu_{\Z}([k])}}$.
Consequently, $\alpha_{i}^{n,k}\in\R^{+}$ and \[
\nu_{\Z}([k\omega j])=\nu_{\Z}([\omega j])\left(\alpha_{i}^{n,k}\right)^{2}\frac{\nu_{\Z}([k])}{\nu_{\Z}([i])}.\]
Now it remains to be shown that $c_{i}^{n,k}$ are independent of
$n\in\N$. For $n\in\N$, $\omega\in\Sigma_{A}^{n}$ with $\omega_{0}=i$
and $k\in\underline{N}$ it follows that \begin{align*}
\nu_{\Z}([k\omega]) & =\sum_{j\in\underline{N}}\nu_{\Z}([k\omega j])=\sum_{j\in\underline{N}}\nu_{\Z}([\omega j])\frac{\left(\alpha_{i}^{n,k}\right)^{2}\nu_{\Z}([k])}{\nu_{\Z}([i])}\\
 & =\nu_{\Z}([\omega])\frac{\left(\alpha_{i}^{n,k}\right)^{2}\nu_{\Z}([k])}{\nu_{\Z}([i])}.\end{align*}
On the other hand we can write $\omega\in\Sigma_{A}^{n}$ with $\omega_{0}=i$
as $\omega=\widetilde{\omega}\omega_{n-1}$ for a suitable $\widetilde{\omega}\in\Sigma_{A}^{n-1}$,
$\widetilde{\omega}_{0}=i$, and so \begin{align*}
\nu_{\Z}([k\omega]) & =\nu_{\Z}([k\widetilde{\omega}\omega_{n-1}])=\nu_{\Z}([\widetilde{\omega}\omega_{n-1}])\frac{\left(\alpha_{i}^{n-1,k}\right)^{2}\nu_{\Z}([k])}{\nu_{\Z}([i])}\\
 & =\nu_{\Z}([\omega])\frac{\left(\alpha_{i}^{n-1,k}\right)^{2}\nu_{\Z}([k])}{\nu_{\Z}([i])}.\end{align*}
Thus, $\alpha_{i}^{n-1,k}=\alpha_{i}^{n,k}$ and so $\alpha_{i}^{n,k}=\alpha_{i}^{m,k}$
for all $n,m\in\N$, $k,i\in\underline{N}$. In the following we write
$\alpha_{i}^{k}$ for $\alpha_{i}^{n,k}$. 

Define $\kappa_{k,i}:=\left(\alpha_{i}^{k}\right)^{2}\nu_{\Z}([k])/\nu_{\Z}([i])$
for $k,i\in\underline{N}$, then we have $\nu_{\Z}([k\omega j])=\kappa_{k,\omega_{0}}\nu_{\Z}([\omega j])$
for all $\omega\in\Sigma_{A}^{*}$, $j,k\in\underline{N}$. From this
property we conclude the Markov relation since for any $k,i\in\underline{N}$
\[
\nu([ki])=\sum_{j\in\underline{N}}\nu([kij])=\sum_{j\in\underline{N}}\kappa_{k,i}\nu([ij])=\kappa_{k,i}\nu([i])\]
and so\[
\nu([ki])=\nu([k])\frac{\kappa_{k,i}\nu([i])}{\nu([k])}.\]
Define $\pi_{ki}:=\kappa_{k,i}\nu([i])/\nu([k])=\left(\alpha_{i}^{k}\right)^{2}$,
then $\pi_{ki}$ is a incidence probability. Consequently, we have
that if a two-sided MRA holds then the measure $\nu$ is Markovian.
The reversed implication will be shown in Section \ref{sub:MRA-for-Markov}.
\end{proof}

\subsection{Mother wavelets for MIM }

In this section we are in the case of Remark \ref{rem:darstellung phi}
(2) and so we consider for each father wavelet $\varphi_{i}$, $i\in\underline{N}$,
a matrix of coefficients; more precisely on each scale we have to
consider for each element of the alphabet $\underline{N}$ a matrix
of coefficients. We slightly change the notation from $c_{j}^{n,k,l}$
to $c_{j}^{\omega,l}$ for $\omega\in\Sigma_{A}^{n}$, since the information
about $n$ and $k$ are coded; $n$ is given by the length of a word
and $k=\sum_{i=0}^{n-1}\omega_{n-1-i}N^{i}$. 

For $\omega\in\Sigma_{A}^{n+1}$ we need a matrix of size $q^{\omega_{n}}\times q^{\omega_{n}}$,
where $q^{\omega_{n}}=\card\{j\in\underline{N}:\, A_{\omega_{n}j}=1\}$.
First we determine $c_{j}^{\omega,k}\in\C$, $j\in\underline{N}$,
$k\in\underline{q^{\omega_{n}}}\backslash\{0\}$, such that the $\left(q^{\omega_{n}}\times q^{\omega_{n}}\right)$-matrix
\[
M_{\omega}:=\left(\begin{array}{c}
\left(\sqrt{\nu_{\Z}([\omega j])}\right)_{j\in D_{\omega_{n}}}\\
\left(A_{\omega_{n}j}c_{j}^{\omega,k}\right)_{k\in\underline{q^{\omega_{n}}}\backslash\{0\},j\in D_{\omega_{n}}}\end{array}\right),\]
where $D_{\omega_{n}}=\left\{ j\in\underline{N}:A_{\omega_{n}j}=1\right\} $
is unitary. This is done as explained above via the Gram-Schmidt process. 

We define for $\omega\in\Sigma_{A}^{n+1}$, $k=\sum_{i=0}^{n}\omega_{n-i}N^{i}$
the basis functions as: for\textbf{ $l\in\underline{q^{\omega_{n}}}\backslash\{0\}$}
\[
\psi^{\omega,l}=U^{(n)}T^{k}\sum_{j\in\underline{N}}A_{\omega_{n}j}c_{j}^{\omega,l}\varphi_{j}.\]

These functions can be written differently for $\omega\in\Sigma_{A}^{n+1}$,
$k\in\underline{q^{\omega_{n}}}\backslash\{0\}$, as 

\[
\psi^{\omega,k}=\sum_{j\in\underline{N}}A_{\omega_{n}j}c_{j}^{\omega,k}\cdot\left(\nu_{\Z}([\omega j])\right)^{-1/2}\cdot\mathbbm{1}_{[\omega j]}.\]
From Theorem \ref{thm:MRA} and Theorem \ref{thm:MRa} the following
corollary follows. 
\begin{cor}
\label{pro:The-orthonormal-basistwo sided}An orthonormal basis for
$L^{2}(\nu_{\Z})$ is given by\[
\left\{ T^{l}\psi^{\omega,k}:\, l\in\Z,\,\omega\in\Sigma_{A}^{*},\, k\in\{1,\dots,q^{\omega_{|\omega|-1}}-1\}\right\} \cup\left\{ T^{l}\varphi_{j}:l\in\Z,j\in\underline{N}\right\} .\]
\end{cor}
\begin{rem}
In fact, the proofs of Theorem \ref{thm:MRA} and Theorem \ref{thm:MRa}
show that we have for $n\in\N$ \[
\cl\spn\left\{ T^{l}\psi^{\omega,k}:\, l\in\Z,\,\omega\in\Sigma_{A}^{n},\, k\in\{1,\dots,q^{\omega_{n-1}}-1\}\right\} =V_{n}\ominus V_{n-1}.\]

\end{rem}

\subsection{\label{sub:MRA-for-Markov}MRA for a Markov measures}

In this section we construct a wavelet basis on the limit set translated
by $\Z$ where the underlying measure $\nu$ is Markovian. For this
fix a probability vector $p=\left(p_{0},p_{1},\dots,p_{N-1}\right)$
and a $\left(N\times N\right)$ stochastic matrix $\Pi=\left(\pi_{jk}\right)_{j,k\in\underline{N}}$
such that for $\omega\in\Sigma_{A}^{n}$ we have \[
\nu([\omega])=p_{\omega_{0}}\prod_{i=0}^{n-2}\pi_{\omega_{i}\omega_{i+1}}.\]
Furthermore, we have that $\pi_{jk}=0$ if $A_{jk}=0$.

This is a special case of the one in the last section. Therefore,
we omit some proofs here and mainly state the results, so that the
differences become clear. 

In this construction we only have to define one operator $U$ since
we obtain $U^{(n)}$ by $U^{n}$, i.e. by iteration of $U$. Another
main difference is that we do not need one matrix for every $\omega\in\Sigma_{A}^{*}$
to obtain the mother wavelets, but we only need matrices for $\omega\in\Sigma_{A}^{1}=\underline{N}$.
So we need not more than $N^{2}$ matrices. This follows from Lemma
\ref{lem:Wn}.

The setting is as defined in Section \ref{sec:Setting-(IFS-with}\textcolor{black}{.}\textcolor{red}{{}
}Set $U:=U^{(1)}$and so it takes the form in (\ref{eq:def U}). By
the Markov property we have $\frac{\nu_{\Z}([i])}{\nu_{\Z}([ji])}=\frac{p_{i}}{p_{j}\pi_{ji}}$
and hence one easily verifies that $U^{(n)}=U^{n}$. Also notice that
$U$ is not unitary\textcolor{red}{{} }\textcolor{black}{unless}\textcolor{red}{{}
}we have that $A_{ij}=1$ for all $i,j\in\underline{N}$.

Now we turn to the form of $U^{*}$.
\begin{lem}
$U^{*}$ has the form\begin{equation}
U^{*}f(x)=\sum_{k\in\Z}\sum_{j\in\underline{N}}\sum_{i\in\underline{N}}\sqrt{\frac{p_{j}\pi_{ji}}{p_{i}}}\cdot\mathbbm{1}_{[i]}(x-j-Nk)\cdot f(\tau_{j}(x-j-Nk)+k).\label{eq:U*}\end{equation}
\end{lem}
\begin{rem}
Notice that $U^{*}=U^{(-1)}$ and $\left(U^{*}\right)^{n}=U^{(-n)}$.\end{rem}
\begin{proof}
To prove that $U^{*}$ has the form above we use the $\Z$-translation
invariance of the measure $\nu_{\Z}$ and the fact that $\frac{d\nu_{\Z}\circ\tau_{j}}{d\nu_{\Z}}=\frac{p_{j}\pi_{ji}}{p_{i}}$
on $[i]$. We obtain this Radon-Nikodym derivative since for a cylinder
set $[\omega]$, $\omega\in\Sigma_{A}^{*}$, we have \begin{align*}
\nu_{\Z}\left(\tau_{j}([\omega])\right)=p_{j}\pi_{j\omega_{0}}\prod_{i=0}^{n}\pi_{\omega_{i}\omega_{i+1}}\end{align*}
and $\nu_{\Z}([\omega])=p_{\omega_{0}}\prod_{i=0}^{n}\pi_{\omega_{i}\omega_{i+1}}$. 

Consequently, we obtain that for $f,g\in L^{2}(\nu_{\Z})$ \begin{align*}
 & \langle Uf|g\rangle\\
= & \int\sum_{k\in\Z}\sum_{j\in\underline{N}}\sum_{i\in\underline{N}}\sqrt{\frac{p_{i}}{p_{j}\pi_{ji}}}\cdot\mathbbm{1}_{[ji]}(x-k)\cdot f(\tau_{j}^{-1}(x-k)+j+Nk)\overline{g(x)}d\nu_{\Z}(x)\\
= & \int\sum_{k\in\Z}\sum_{j\in\underline{N}}\sum_{i\in\underline{N}}\sqrt{\frac{p_{i}}{p_{j}\pi_{ji}}}\cdot\mathbbm{1}_{[ji]}(x)\cdot f(\tau_{j}^{-1}(x)+j+Nk)\overline{g(x+k)}d\nu_{\Z}(x)\\
= & \int\sum_{k\in\Z}\sum_{j\in\underline{N}}\sum_{i\in\underline{N}}\sqrt{\frac{p_{i}}{p_{j}\pi_{ji}}}\cdot\mathbbm{1}_{[ji]}(\tau_{j}(x))\cdot f((x)+j+Nk)\overline{g(\tau_{j}(x)+k)}d\nu_{\Z}(\tau_{j}(x))\\
= & \int\sum_{k\in\Z}\sum_{j\in\underline{N}}\sum_{i\in\underline{N}}\sqrt{\frac{p_{i}}{p_{j}\pi_{ji}}}\cdot\mathbbm{1}_{[ji]}(\tau_{j}(x))\cdot f((x)+j+Nk)\overline{g(\tau_{j}(x)+k)}\cdot\frac{p_{j}\pi_{ji}}{p_{i}}\cdot d\nu_{\Z}(x)\\
= & \int f(x)\sum_{k\in\Z}\sum_{j\in\underline{N}}\sum_{i\in\underline{N}}\sqrt{\frac{p_{j}\pi_{ji}}{p_{i}}}\cdot\mathbbm{1}_{[i]}(x-j-Nk)\cdot\overline{g(\tau_{j}(x-j-Nk)+k)}d\nu_{\Z}(x)\\
= & \langle f|U^{*}g\rangle,\end{align*}
with $U^{*}g$ as in (\ref{eq:U*}).
\end{proof}
Now we turn to the definition of the father wavelets which we use
in the MRA. Define the $N$ \textit{father wavelets} as $\varphi_{i}=\left(\nu_{\Z}([i])\right)^{-1/2}\mathbbm{1}_{[i]}$
for $i\in\underline{N}$. 
\begin{rem}
Notice that the family of father wavelets $\left(\varphi_{i}\right)_{i\in\underline{N}}$
is orthonormal by definition.
\end{rem}
Now we turn to the properties of the operators $U$ and $T$ given
in Proposition \ref{pro:eigenschaften UT-1}.
\begin{proof}
[Proof of Proposition  \ref{pro:eigenschaften UT-1}]We have that
(\ref{enu:Prop U (1)}), (\ref{enu:,Prop U (2)}), (\ref{enu:,Prop U (3)})
and (\ref{enu:,Prop U (4)}) follow directly from Proposition \ref{pro:The-operatorsUT-1}
since it is a special case of $U^{(n)}$ in the section above. 

ad (\ref{enu:-Prop U (5)}): This proof is analogous to the one of
(\ref{enu:,Prop U (4)}) or Proposition \ref{pro:The-operatorsUT-1}
(\ref{enu:(5)UZ-1}). We obtain for $f\in L^{2}(\nu_{\Z})$ \begin{align*}
U^{*}Uf(x) & =\sum_{k\in\Z}\sum_{j\in\underline{N}}A_{ji}\mathbbm{1}_{[i]}(x-j-Nk)\cdot f(x).\end{align*}
$\sum_{k\in\Z}\sum_{j\in\underline{N}}\sum_{i\in\underline{N}}A_{ji}\mathbbm{1}_{[i]}(x-j-Nk)=1$
for all $x\in\R$ if and only if $A_{ji}=1$ for all $i,j\in\underline{N}$. 
\end{proof}
Now we turn to the proof of the backward direction of Theorem \ref{thm:MRA-1}.
Some of the properties follow directly from the proof of Theorem \ref{thm:MRa}.
\begin{proof}
[Proof of Theorem \ref{thm:MRA-1} ''$\ensuremath{\Longleftarrow}$'']
We show the properties (\ref{enu:DefMRA1-1}) to (\ref{enu:DefMRA6-1})
of Definition \ref{def:MRA allgemein}. The property (\ref{enu:DefMRA2-1})
follows from Theorem \ref{thm:MRa}. 

ad (\ref{enu:-DefMRA5-1}): For $n\in\N_{0}$ it follows directly
from Theorem \ref{thm:MRa}. For $n\in\Z$, $n<0$, $x\in\R$, $k\in\underline{N}$,
it follows by \begin{align*}
 & \left(U^{*}\right)^{|n|}\varphi_{k}(x)\\
= & \sum_{\omega\in\Sigma_{A}^{|n|}:\omega_{0}=k}\sum_{i\in\underline{N}}\sqrt{\prod_{l=1}^{|n|-2}\pi_{\omega_{l},\omega_{l+1}}\cdot\pi_{k,\omega_{1}}\pi_{\omega_{|n|-1},i}}\varphi_{i}\left(x-\sum_{l=0}^{|n|-1}\omega_{|n|-1-l}N^{l}\right)\\
= & \sum_{j\in\underline{N}}\sqrt{\pi_{k,j}}\Bigg(\sum_{\omega\in\Sigma_{A}^{|n|-1}:\omega_{0}=j}\sum_{i\in\underline{N}}\sqrt{\prod_{l=1}^{|n|-3}\pi_{\omega_{l},\omega_{l+1}}\cdot\pi_{j,\omega_{1}}\pi_{\omega_{|n|-2},i}}\\
 & \ \ \ \varphi_{i}\left(x-\sum_{l=0}^{|n|-2}\omega_{|n|-2-l}N^{l}-kN^{|n|-1}\right)\Bigg)\\
= & \sum_{j\in\underline{N}}\sqrt{\pi_{k,j}}\left(U^{*}\right)^{|n|-1}T^{k}\varphi_{j}(x).\end{align*}

ad (\ref{enu:DefMRA1-1}): For $n\in\N_{0}$ it follows directly from
Theorem \ref{thm:MRa}. For $n\in\Z$, $n<0$, $k\in\underline{N}$,
it follows from \[
\left(U^{*}\right)^{|n|}\varphi_{k}=\sum_{j\in\underline{N}}\sqrt{\pi_{k,j}}\left(U^{*}\right)^{|n|-1}T^{k}\varphi_{j}.\]

ad (\ref{enu:DefMRA3-1}): We have that $\bigcap_{n\in\Z}V_{n}=\left\{ 0\right\} ,$
because the support of $\left(U^{*}\right)^{n}\varphi_{j}$, $j\in\underline{N}$,
grows in $n\in\N$. More precisely, for $j\in\underline{N}$ \[
\nu_{\Z}\left(\supp\left(\left(U^{*}\right)^{n}\varphi_{j}\right)\right)=\sum_{i\in\underline{N}}\nu_{\Z}\left([i]\right)\left(\card\left\{ \omega\in\Sigma_{A}^{n+1}:\omega_{0}=j,\omega_{n}=i\right\} \right).\]
Consequently, $\{0\}=\bigcap_{j\in\Z}V_{j}$ since any function $f\in\bigcap_{j\in\Z}V_{j}$
must be constant for every $n\in\N$ on $\supp\left(\left(U^{*}\right)^{n}\varphi_{j}\right)$,
for $j\in\underline{N}$.

ad (\ref{enu:DefMRA4-1}): This property follows directly from the
definition of the spaces $V_{j}$ and Proposition \ref{pro:The-operatorsUT-1}
(\ref{enu:(4)TU-1}) with the observation that $U^{(n)}=U^{n}$ and
$U^{(-n)}=\left(U^{*}\right)^{n}$, $n\in\N_{0}$.

ad (\ref{enu:DefMRA6-1}): This property follows from Proposition
\ref{pro:eigenschaften UT-1} (\ref{enu:,Prop U (4)}) and (\ref{enu:-Prop U (5)}).\end{proof}
\begin{rem}
Now we give some remarks concerning the father wavelets. 
\begin{enumerate}
\item The relation for the functions $\varphi_{i}$, $i\in\underline{N}$,
can also be written as\[
\left(\varphi_{j}\right)_{j\in\underline{N}}^{t}=\sum_{l\in\underline{N}}M_{l}\left(UT^{l}\varphi_{j}\right)_{j\in\underline{N}},\]
where the $M_{l}$ are $\left(N\times N\right)$-matrices with $\left(M_{l}\right)_{n,k}=\begin{cases}
\sqrt{\pi_{lk}}, & n=l,\\
0, & \text{else},\end{cases}$ for $n,k\in\underline{N}$.
\item Notice that for $k\in\Z$ we can write $k=a_{0}+Nl$, where $a_{0}\in\underline{N}$
and some $l\in\Z$, i.e. $k$ is in the $N$-adic expansion. Then
we obtain \begin{eqnarray*}
UT^{k}\varphi_{j} & = & \begin{cases}
0, & \text{if }A_{a_{0}j}=0,\\
\left(p_{a_{0}}\cdot\pi_{a_{0}j}\right)^{-1/2}T^{l}\mathbbm{1}_{[a_{0}j]}, & \text{else}.\end{cases}\end{eqnarray*}

\item Notice that in $\left\{ U^{n}T^{k}\varphi_{i}:\, n\in\N,\, k\in\Z,\, i\in\underline{N}\right\} $
some functions are constantly zero. These functions are precisely
those where for $k\in\Z$ written in the $N$-adic expansion, $k=\sum_{j=0}^{n-1}k_{n-1-j}N^{i}+lN^{n}$,
$k_{j}\in\underline{N}$, $l\in\Z$, either $A_{k_{j}k_{j+1}}=0$
for some $j\in\{0,\dots,n-2\}$ or $A_{k_{n-1}i}=0$.
\end{enumerate}
\end{rem}

\subsubsection{Mother wavelets for Markov measures}

The construction of the mother wavelets simplifies in this setting
because we only have to consider mother wavelets for one scale and
obtain the other by iterative application of the operators $U$ and
$T$ by Lemma \ref{lem:Wn}. The mother wavelets are constructed via
$N$ matrices as given in Lemma \ref{def:mother wavelets markov-1}
and so the mother wavelets are defined for $k\in\underline{N}$ and
$l\in\underline{q^{k}}\backslash\{0\}$, by\[
\psi^{k,l}=UT^{k}\sum_{j\in\underline{N}}A_{kj}c_{j}^{k,l}\varphi_{j}\]
 for coefficients $c_{j}^{k,l}\in\mathbb{C}$ as in Lemma \ref{def:mother wavelets markov-1}.
\begin{rem}
$\ $
\begin{enumerate}
\item The number of mother wavelets we obtain is $\sum_{k\in\underline{N}}q^{k}\leq N^{2}$.
In the case of $N^{2}$ mother wavelets we are back in the case of
fractals given by an IFS.
\item Notice that $\sum_{l=1}^{q^{k}-1}A_{ki}A_{kj}c_{i}^{k,l}\overline{c}_{j}^{k,l}+\sqrt{\pi_{ki}}\sqrt{\pi_{kj}}=\delta_{i,j}$.
\item Alternatively we can define the mother wavelets as the elements of
the vector \[
\left(\psi^{k,l}\right)_{l\in\{1,\dots,q^{k}-1\}}^{t}=\left(\left(A_{kj}c_{j}^{k,l}\right)_{l\in\underline{q^{k}}\backslash\{0\},j\in\underline{N}}\right)\left(UT^{k}\varphi_{j}\right)_{j\in\underline{N}}^{t}.\]

\item Here we can see that we only need mother wavelets for $W_{0}$ since
\[
\sum_{j\in\underline{N}}A_{kj}c_{j}^{k,i}\left(\nu_{\Z}([\omega j])\right)^{1/2}=\sqrt{p_{\omega_{0}}\prod_{i=1}^{n-2}\pi_{i(i+1)}}\sum_{j\in\underline{N}}A_{kj}c_{j}^{k,i}\sqrt{\pi_{\omega_{n-1}j}}=0,\]
which was the crucial condition in the case of the last section. 
\end{enumerate}
\end{rem}
\begin{cor}
\begin{align*}
 & \left\{ U^{n}T^{m}\psi^{k,l}:\, n\in\N_{0},\, m\in D_{n,k},\, k\in\underline{N},\, l\in\underline{q^{k}}\backslash\{0\}\right\} \\
\cup & \left\{ \left(U^{*}\right)^{n}T^{m}\psi^{k,l}:\, n\in\N,\, m\in\Z,\, k\in\underline{N},\, l\in\underline{q^{k}}\backslash\{0\}\right\} \\
\cup & \left\{ \left(U^{*}\right)^{n}T^{k}\varphi_{j}:\, n\in\N,\, k\in N\Z+l,\, j,l\in\underline{N},\, A_{jl}=0\right\} \end{align*}
gives an ONB for $L^{2}(\nu_{\Z})$, where \begin{align*}
D_{n,k}= & \Bigg\{ m\in\Z:\, m=\sum_{i=0}^{n-1}\omega_{n-1-i}N^{i}+N^{n}l,\omega_{i}\in\underline{N},\,(\omega_{0},\dots,\omega_{n-1})\in\Sigma_{A}^{n}\\
 & \ \ \text{ and }A_{\omega_{0}k}=1,\, l\in\Z\Bigg\}.\end{align*}
\end{cor}
\begin{rem}
$\ $
\begin{enumerate}
\item Because of $UW_{-1}=W_{0}$ we only have to add those functions $T^{k}\varphi_{j}$,
$k\in\Z$, $j\in\underline{N}$, with $UT^{k}\varphi_{j}=0$ to the
basis of $U^{*}\left(W_{0}\right)$ to obtain a basis of $W_{-1}$. 
\item Notice that
\end{enumerate}
\begin{eqnarray*}
\psi^{k,l} & = & UT^{k}\sum_{i\in\underline{N}}A_{ki}c_{i}^{k,l}\varphi_{i}\\
 & = & \sum_{i\in\underline{N}}A_{ki}c_{i}^{k,l}\cdot\left(p_{k}\cdot\pi_{ki}\right)^{-1/2}\cdot\mathbbm{1}_{[ki]}.\end{eqnarray*}

\end{rem}

\subsection{\label{sub:Examples}Examples}

In the construction of \cite{MaPa09} only Cantor sets with incidence
matrix are considered, i.e. the IFS has the form $\left(\tau_{i}(x)=\frac{x+i}{N}\right)_{i\in\underline{N}}$,
and there exists a incidence matrix $A$. The limit set has then the
Hausdorff dimension $\delta=\dim_{H}(X)=\frac{\log r(A)}{\log N}$,
where $r(A)$ is the spectral radius of $A$. So we consider the $\delta$-dimensional
Hausdorff measure $\mu$ restricted to the\textbf{ }by\textbf{ $\Z$
}translated set $X$. It follows that $p_{j}=\mu([j])$ and $\pi_{ij}=\frac{N^{-2\delta}p_{j}}{p_{i}}$.
Consequently, in this case we can rewrite our conditions for obtaining
the coefficients of the mother wavelets in a simpler way. More precisely,
for $k\in\underline{N}$ instead of \[
\sum_{j\in\underline{N}}A_{kj}c_{j}^{k,i}\sqrt{\pi_{kj}}=0\]
 we obtain the condition \[
\sum_{j\in\underline{N}}A_{kj}c_{j}^{k,i}\sqrt{p_{j}}=0.\]

Although the basis in \cite{MaPa09} is only given in terms of the
representation of a Cuntz-Krieger algebra we can now give a scaling
operator $U$ in the sense of (\ref{eq:def U}) for this case. More
precisely, we obtain\[
Uf(x)=N^{\delta}\sum_{k\in\Z}\sum_{j\in\underline{N}}\mathbbm{1}_{[j]}(x-k)\cdot f(\tau_{j}^{-1}(x-k)+j+Nk).\]

\textbf{\textcolor{black}{Proof of Example \ref{exa:-Transformation}:
}}We clearly have that the $\beta$-transformation belongs to the
class of Markov measures. Consequently, we have that $\left(\mu,U,T\right)$
allows a MRA. We can construct the mother wavelets along the lines
of Section \ref{sub:MRA-for-Markov}. Since we have that in this case
$d_{0}=2$ and $d_{1}=1$ we only have to construct coefficients for
$\varphi_{0}$ to obtain the mother wavelets. These coefficients are
given in the following matrix which is unitary:\[
\left(\begin{array}{cc}
\sqrt{\beta-1} & \sqrt{2-\beta}\\
\sqrt{2-\beta} & -\sqrt{\beta-1}\end{array}\right).\]
Thus, the mother wavelet is $\psi=U\left(\sqrt{2-\beta}\varphi_{0}-\sqrt{\beta-1}\varphi_{1}\right)$.
To obtain the basis we further notice that $UT\varphi_{1}=0$ and
so we have to keep $T^{k}\varphi_{1}$, $k\in2\Z+1$ in the basis.

\section{\label{sec:Operator-algebra}Operator algebra}

In the case of one father wavelet we obtain a so-called low-pass filter
function and high-pass filter functions, in terms of which the mother
wavelets are given. Via these filter functions we obtain a representation
of the Cuntz algebra $\mathcal{O}_{N}$, where $N$ is the number
of filter functions. In the case of multiwavelets we can obtain weaker
relations. Here we restrict to the case of MIM with underlying Markov
measure as treated in Section \ref{sub:MRA-for-Markov}. These results
are in correspondence to results in \cite{BFMP10}.

The relations for the father and the mother wavelets can be written
in the following way:

For the following we introduce for $z\in\mathbb{T}:=\left\{ \omega\in\C:|\omega|=1\right\} $
the low-pass filter \[
H(z)=\left(\sqrt{\pi_{kl}}z^{k}\right)_{l,k\in\underline{N}}\]
and for each $k\in\underline{N}$ and $z\in\mathbb{T}$ the high-pass
filter \[
G_{k}(z)=\left(A_{kl}c_{l}^{k,j}z^{k}\right)_{j\in\underline{q^{k}}\backslash\{0\},l\in\underline{N}}.\]

With these definitions we obtain the following immediate lemma.
\begin{lem}
Let $\phi=\left(\varphi_{j}\right)_{j\in\underline{N}}^{t}$, then
$\phi=UH(T)\phi$ and let $\psi_{k}=\left(\psi^{k,j}\right)_{j\in\underline{q^{k}}\backslash\{0\}}^{t}$
for $k\in\underline{N}$, then $\psi_{k}=UG_{k}(T)\phi$, where the
operators $U$ and $T$ are applied to evey entry in the vector. \end{lem}
\begin{rem}
It follows that for $z\in\mathbb{T}$ \[
\overline{H}(z)H^{t}(z)=\left(\sum_{j\in\underline{N}}\sqrt{\pi_{kj}\pi_{lj}}z^{l-k}\right)_{k,l\in\underline{N}}\]
and for $k\in\underline{N}$, $z\in\mathbb{T}$,

\begin{eqnarray*}
\overline{G_{k}}(z)G_{k}^{t}(z) & = & I.\end{eqnarray*}

\end{rem}
These filter functions lead us to the definitions of certain {}``isometries''. 
\begin{defn}
For $z\in\mathbb{T}$ and $f=\left(f_{0},\dots,f_{N-1}\right)$, $f_{j}\in L^{2}(\mathbb{T},\lambda)$,
define \[
S_{H}f(z)=\sqrt{N}H^{t}(z)f\left(z^{N}\right)\]
 and for $k\in\underline{N}$, $z\in\mathbb{T},$\[
S_{G_{k}}f(z)=G_{k}^{t}(z)f\left(z^{N}\right).\]

\end{defn}
For these {}``isometries'' we have the following properties. 
\begin{prop}
The following relations hold:
\begin{enumerate}
\item $S_{H}^{*}S_{H}=I$, \label{enu:S1}
\item $S_{G_{k}}^{*}S_{G_{k}}=I$, $k\in\underline{N}$,\label{enu:S2}
\item $S_{H}^{*}S_{G_{k}}=0$ and $S_{G_{k}}^{*}S_{H}=0$, $k\in\underline{N}$,\label{enu:S3}
\item $S_{G_{i}}^{*}S_{G_{j}}=0$, $i,j\in\underline{N}$, $i\neq j$ .\label{enu:S4}
\end{enumerate}
\end{prop}
\begin{rem}
Realize that for $z\in\mathbb{T}$, $f=\left(f_{0},\dots,f_{N-1}\right)$,
$f_{j}\in L^{2}(\mathbb{T},\lambda)$, \[
S_{H}^{*}f(z)=\frac{1}{\sqrt{N}}\sum_{\omega^{N}=z}\overline{H}(\omega)f(\omega)\]
 and for $k\in\underline{N}$, $z\in\mathbb{T}$, $f=\left(f_{0},\dots,f_{N-1}\right)$,
$f_{j}\in L^{2}(\mathbb{T},\lambda)$, \[
S_{G_{k}}^{*}f(z)=\frac{1}{N}\sum_{\omega^{N}=z}\overline{G_{k}}(\omega)f(\omega).\]
\end{rem}
\begin{proof}
ad (\ref{enu:S1}): Let $z\in\mathbb{T}$, $f=\left(f_{0},\dots,f_{N-1}\right)$,
$f_{j}\in L^{2}(\mathbb{T})$, then\begin{eqnarray*}
S_{H}^{*}S_{H}f(z) & = & \sum_{\omega^{N}=z}\overline{H}(\omega)H^{t}(\omega)f(\omega^{N})\\
 & = & \sum_{\omega^{N}=z}\overline{H}(\omega)H^{t}(\omega)f(z)=f(z)\end{eqnarray*}

ad (\ref{enu:S2}): Let $k\in\underline{N}$, $z\in\mathbb{T}$, $f=\left(f_{0},\dots,f_{N-1}\right)$,
$f_{j}\in L^{2}(\mathbb{T})$, then \begin{eqnarray*}
S_{G_{K}}^{*}S_{G_{k}}f(z) & = & \frac{1}{N}\sum_{\omega^{N}=z}\overline{G_{k}}(\omega)G_{k}^{t}(\omega)f(z)=f(z)\end{eqnarray*}

ad (\ref{enu:S3}): Let $k\in\underline{N}$, $z\in\mathbb{T}$, $f=\left(f_{0},\dots,f_{N-1}\right)$,
$f_{j}\in L^{2}(\mathbb{T})$, then \begin{eqnarray*}
S_{H}^{*}S_{G_{k}}f(z) & = & \frac{1}{N}\sum_{\omega^{N}=z}\overline{H}(\omega)G_{k}^{t}(\omega)f(z)=0,\end{eqnarray*}
since $\sum_{\omega^{N}=z}\overline{H}(\omega)G_{k}^{t}(\omega)=0$
by summing up the roots of unity.

For $S_{G_{k}}^{*}S_{H}$ we use that $\overline{G_{k}}(\omega)H^{t}(\omega)=0$
by the choice of the coefficients $c_{j}^{k,l}$. 

ad (\ref{enu:S4}): Let $i,j\in\underline{N}$, $i\neq j$, $z\in\mathbb{T}$,
$f=\left(f_{0},\dots,f_{N-1}\right)$, $f_{j}\in L^{2}(\mathbb{T})$,
then\begin{align*}
S_{G_{i}}^{*}S_{G_{j}}f(z) & =\frac{1}{N}\sum_{\omega^{N}=z}\overline{G_{i}}(\omega)G_{j}^{t}(\omega)f(z)=0,\end{align*}
by summing up the roots of unity. 
\end{proof}
Here we have seen that in contrast to the filter functions for a usual
MRA with one father wavelet and a unitary scaling operator $U$, we
do not obtain a representation of a Cuntz algebra since we do not
neccessarily have that $S_{H}S_{H}^{*}+\sum_{k\in\underline{N}\backslash\{0\}}S_{G_{k}}S_{G_{K}}^{*}=I$.
So we only obtain weaker relations between these filter functions. 
\providecommand{\bysame}{\leavevmode\hbox to3em{\hrulefill}\thinspace}
\providecommand{\MR}{\relax\ifhmode\unskip\space\fi MR }
\providecommand{\MRhref}[2]{%
  \href{http://www.ams.org/mathscinet-getitem?mr=#1}{#2}
}
\providecommand{\href}[2]{#2}


\begin{thebibliography}{BFMP10}

\bibitem[Alp93]{Al93}
Bradley~K. Alpert, \emph{A class of bases in {$L^2$} for the sparse
  representation of integral operators}, SIAM J. Math. Anal. \textbf{24}
  (1993), no.~1, 246--262. \MR{1199538 (93k:65104)}

\bibitem[BFMP10]{BFMP10}
Lawrence~W. Baggett, Veronika Furst, Kathy~D. Merrill, and Judith~A. Packer,
  \emph{Classification of generalized multiresolution analyses}, J. Funct.
  Anal. \textbf{258} (2010), no.~12, 4210--4228. \MR{2609543}

\bibitem[BK10]{BoKe10}
Jana Bohnstengel and Marc Kesseb{\"o}hmer, \emph{Wavelets for iterated function
  systems}, J. Funct. Anal. \textbf{259} (2010), no.~3, 583--601. \MR{2644098}

\bibitem[Bod07]{Bod06}
Mats Bodin, \emph{Wavelets and {B}esov spaces on {M}auldin-{W}illiams
  fractals}, Real Anal. Exchange \textbf{32} (2006/07), no.~1, 119--143.
  \MR{2329226 (2008h:42063)}

\bibitem[CK03]{CK03}
Mark Crovella and Eric Kolaczyk, \emph{Graph wavelets for spatial traffic
  analysis}, Proceedings of IEEE Infocom, April 2003.

\bibitem[Dau92]{Dau92}
Ingrid Daubechies, \emph{Ten lectures on wavelets}, CBMS-NSF Regional
  Conference Series in Applied Mathematics, vol.~61, Society for Industrial and
  Applied Mathematics (SIAM), Philadelphia, PA, 1992. \MR{1162107 (93e:42045)}

\bibitem[DJ06]{DJ03}
Dorin~E. Dutkay and Palle E.~T. Jorgensen, \emph{Wavelets on fractals}, Rev.
  Mat. Iberoam. \textbf{22} (2006), no.~1, 131--180. \MR{2268116 (2008h:42071)}

\bibitem[GP96]{GP96}
Jean-Pierre Gazeau and Jiri Patera, \emph{Tau-wavelets of {H}aar}, J. Phys. A
  \textbf{29} (1996), no.~15, 4549--4559. \MR{1413218 (97f:42054)}

\bibitem[Jon98]{Jo98}
Alf Jonsson, \emph{Wavelets on fractals and {B}esov spaces}, Journal of Fourier
  Analysis and Applications \textbf{4} (1998), 329--340, 10.1007/BF02476031.

\bibitem[KS10]{KeSa10}
Marc {Kesseb{\"o}hmer} and Tony {Samuel}, \emph{{Spectral metric spaces for
  Gibbs measures}}, ArXiv:1012.5152 (2010).

\bibitem[KSS07]{KeStaStr07}
Marc Kesseb{\"o}hmer, Manuel Stadlbauer, and Bernd Stratmann, \emph{Lyapunov
  spectra for {KMS} states on {C}untz - {K}rieger algebras}, Mathematische
  Zeitschrift \textbf{256} (2007), 871--893, 10.1007/s00209-007-0110-y.

\bibitem[MP09]{MaPa09}
Matilde Marcolli and Anna Paolucci, \emph{{C}untz - {K}rieger algebras and
  wavelets on fractals}, Complex Analysis and Operator Theory (2009), 1--41,
  10.1007/s11785-009-0044-y.

\bibitem[MU03]{MU03}
Daniel Mauldin and Mariusz Urba{\'n}ski, \emph{Graph directed {M}arkov
  systems}, Cambridge Tracts in Mathematics, vol. 148, Cambridge University
  Press, Cambridge, 2003, Geometry and dynamics of limit sets. \MR{2003772
  (2006e:37036)}

\bibitem[Par60]{Pa60}
William Parry, \emph{On the {$\beta $}-expansions of real numbers}, Acta Math.
  Acad. Sci. Hungar. \textbf{11} (1960), 401--416. \MR{0142719 (26 \#288)}

\bibitem[R{\'e}n57]{Re57}
Alfr{\'e}d R{\'e}nyi, \emph{Representations for real numbers and their ergodic
  properties}, Acta Math. Acad. Sci. Hungar \textbf{8} (1957), 477--493.
  \MR{0097374 (20 \#3843)}

\end{thebibliography}
\end{document}